\documentclass[english,11pt,a4paper,leqno]{amsart}
\usepackage{amsmath,amstext, amsthm, amssymb}
\usepackage{appendix}

\usepackage{color}
\definecolor{liens}{rgb}{1,0,0}
\usepackage[colorlinks=true, linkcolor=blue, 
hyperfootnotes=true,citecolor=blue,urlcolor=black]{hyperref}

\usepackage{latexsym, amsthm, amsfonts, amsmath,amssymb,graphicx,xcolor}
\input xy
\xyoption{all}
\usepackage[latin1]{inputenc} 
\usepackage[T1]{fontenc} 
\usepackage[english]{babel} 

\newtheorem{theo}{Theorem}[section]

\newtheorem{lem}[theo]{Lemma}
\newtheorem{propo}[theo]{Proposition}
\newtheorem{coro}[theo]{Corollary}

\theoremstyle{definition}
\newtheorem{defi}[theo]{Definition}

\numberwithin{equation}{section}
\theoremstyle{remark}
\newtheorem{rem}[theo]{Remark}

\def\det{\mathrm{det}}

\def\Gal{\mathrm{Gal}}
\def\<{\langle}
\def\>{\rangle}
\def\cM{\mathcal{M}}

\def\Gl{\mathrm{GL}}
\def\GL{\mathrm{GL}}

\def\trd{\mathrm{trdeg}}
\def\tr{\partial}
\def\cL{\mathcal{L}}

\def\Z{\mathbb{Z}}
\def\C{\mathbb{C}}
\def\N{\mathbb{N}}
\def\Q{\mathbb{Q}}
\def\cL{\mathcal{L}}

\def\n{\eta}

\def\d{\delta}

\def\n'{\nu}
\def\d{\delta}

\def\bM{\textbf{M}}

\def\tK{\widetilde{K}}
\def\tC{\widetilde{C}}
\def\cM{\mathcal{M}}
\def\bQ{\mathbf{Q}}
\def\cQ{\mathcal{Q}}
\def\bS{\mathbf{S}}
\def\tG{\widetilde{G}}
\def\tcQ{\widetilde{\mathcal{Q}}}

\def\cN{\mathcal{N}}
\def\dim{\mathrm{dim}}
\def\bS{\textbf{S}}
\def\bQ{\textbf{Q}}
\def\bM{\textbf{M}}
\def\tL{\widetilde{L}}
\def\cD{\mathcal{D}}

\begin{document}

\sloppy

\title{Hypertranscendence and  linear difference equations}

\author{Boris Adamczewski}
\address{Univ Lyon, Universit\'e Claude Bernard Lyon 1, CNRS UMR 5208, Institut Camille Jordan, 43 blvd. du 11 novembre 1918, F-69622 Villeurbanne cedex, France}
\email{boris.adamczewski@math.cnrs.fr}
\author{Thomas Dreyfus}
\address{Institut de Recherche Math\'ematique Avanc\'ee, U.M.R. 7501 Universit\'e de Strasbourg et C.N.R.S. 7, rue Ren\'e Descartes 67084 Strasbourg, France}
\email{dreyfus@math.unistra.fr}
\author{Charlotte Hardouin}
\address{Institut de Mathématiques de Toulouse, Université Paul Sabatier, 118, route de Narbonne, 31062 Toulouse, France}
\email{hardouin@math.univ-toulouse.fr}

\thanks{ This project has received funding from the European Research Council (ERC) under the European Union's Horizon 2020 research and innovation program under the Grant Agreement No 648132.}
\date{\today}

\begin{abstract} 
After H\"older proved his classical theorem about the Gamma function,  
there has been a whole bunch of results showing that solutions to linear difference equations tend to be 
hypertranscendental ({\it i.e.}, they cannot be solution to an algebraic differential equation).  
In this paper, we obtain the first complete results for solutions to  
general linear difference equations associated with the shift operator 
$x\mapsto x+h$ ($h\in\C^*$), the $q$-difference operator $x\mapsto qx$ ($q\in\C^*$ not a root of unity), and the Mahler operator $x\mapsto x^p$ ($p\geq 2$ integer). The only restriction is that we constrain our solutions to be expressed as (possibly ramified) Laurent series in the variable $x$ with complex coefficients (or in the variable $1/x$ in some special case associated with the shift operator). Our proof is based on the parametrized difference Galois theory initiated by Hardouin and Singer. 
We also deduce from our main result a general statement about algebraic independence of values of Mahler 
functions and their derivatives at algebraic points.
\end{abstract} 
\maketitle

\section{Introduction}

Number theorists have to face the following challenge. On the one hand, the fields of rational and algebraic numbers 
are too poor, while, on the other hand, the fields of real and complex numbers seem far too large (uncountable). 
An attempt to classify transcendental numbers leads 
to the introduction of the ring of periods 
$\mathcal P$ (see \cite{KZ}) and to the following classification: 
$$
\mathbb Q \subset \overline{\mathbb Q}\subset \mathcal P\subset  \C \,.
$$
The main features of $\mathcal P$ are that most of classical mathematical constants 
(\emph{e.g.}, $\pi$, $\log 2$, $\zeta(3)$, $\Gamma(1/3)$...)\footnote{However, there should be some exceptions. 
For instance, it is conjectured that  $e$ is not a period.} belong to it, 
and that $\mathcal P$ is a countable set whose elements 
contain only a finite amount of information and can be identified in an algorithmic way. 
After the pioneering works of Cauchy and Riemann, analytic and holomorphic functions 
figure prominently in mathematics. 
As already observed by Hilbert \cite{Hi}, their study gives rise to a similar difficulty: 
the fields of rational and algebraic functions are too limited, but the ring $\C[[x]]$ is huge. 
In order to overcome this deficiency, he suggested to classify transcendental functions according to whether they  
are \emph{holonomic} (\emph{i.e.}, satisfy a linear differential equation with polynomial coefficients), \emph{differentially algebraic} (\emph{i.e.}, satisfy an algebraic differential equation), or not. Hence, we obtain the following classification: 
\begin{equation}\label{eq:class}
\mbox{Rat} \subset \mbox{Alg} \subset \mbox{Hol} \subset \mbox{Dif} \subset \C[[x]]\, .
\end{equation}
Functions that are not differentially algebraic are said to be \emph{hypertranscendental}. 
Note that, if we restrict our attention to power series with algebraic coefficients, 
the rings Hol and Dif become countable, and, again, their elements  
contain only a finite amount of information and can be identified in an algorithmic way. 
These rings play, in the setting of holomorphic functions, a role similar to the one played by $\mathcal P$ 
 in the setting of complex numbers. 
The classification \eqref{eq:class} is also significant in enumerative combinatorics.   
To some extent, the nature of a generating series reflects the underlying structure of the objects it counts  
  (see the discussion in \cite{BM}). 
An appealing example, where all cases of \eqref{eq:class} appear at once, 
is given by the study of generating series associated 
with lattice walks (see, for instance, \cite{DHRS} and the references therein). 

Hilbert also observed that the class of differentially algebraic functions misses important functions 
 coming from number theory.  Here are few examples. The very first comes from a classical result of H\"older \cite{Ho} 
 stating that the gamma function $\Gamma(x)$ is hypertranscendental. Furthermore, 
 it follows from the identity 
 $$
 \zeta(x)=2(2\pi)^{x-1}\Gamma(1-x)\sin\left(\frac{\pi x}{2}\right)\zeta(1-x) 
 $$
that the Riemann zeta function is also hypertranscendental.   
In an other direction, Moore \cite{Mo} 
shown that $\mathfrak f_1(x)=\sum_{n\geq 0}x^{2^n}$ is hypertranscendental. 
A result reproved and extended later by Mahler \cite{Ma30} in connection with the so-called Mahler's 
method in transcendental number theory.  More recently, 
Hardouin and Singer \cite{HS08}  proved hypertranscendence for some $q$-hypergeometric series, such as 
$$
\mathfrak f_2(x):= 
\sum_{n=0}^\infty\frac{(1-a)^2(1-aq)^2\cdots(1-aq^{n-1})^2}{(1-q)^2(1-q^2)^2\cdots(1-q^{n})^2}x^n\, ,
$$
where $q\in\C^*$ is not a root of unity, $a\not\in q^{\mathbb Z}$ and $a^2\in q^{\mathbb Z}$.  
Most proofs about hypertranscendence turn out to be based on the fact that the function under consideration 
satisfies a functional equation of a different type (\emph{i.e.}, not differential).    
For instance, the three results we just mentioned are respectively based on the following linear difference 
equations: 
$$
\Gamma(x+1)=x\Gamma(x) \;, \; \mathfrak f_1(x^2)=\mathfrak f_1(x) - x\;, 
$$
and
$$
\mathfrak f_2(q^2x)-\frac{2ax-2}{a^2x-1}\mathfrak f_2(qx)+\frac{x-1}{a^2x-1}\mathfrak f_2(x)=0\,.
$$
Broadly speaking, satisfying both a linear difference equation and an algebraic differential equation would enforce 
too much symmetry for a transcendental function. 
In this paper, we follow this motto: 

\medskip

\noindent (A) \emph{ A solution to a linear difference equation should be either rational or hypertranscendental. }

\medskip

\noindent Of course, this must be taken with a pinch of salt. For instance, 
$\log(x)$, $\exp(x)$, and the Jacobi theta function $\theta_q(x)=\sum_{n\in\mathbb Z} q^{-n(n-1)/2}x^n$ are all 
differentially algebraic, but they also satisfy the following simple linear difference equations: 
$$
\log (x^p)=p\log x \,, \;\;\exp(x+1)=e\exp(x) \,,\;\; \theta_q(qx)=qx\theta_q(x) \,.
$$
Nevertheless, we obtain in this paper the first complete results about hypertranscendence of solutions to  
general linear difference equations associated with the shift operator 
$x\mapsto x+h$ ($h\in\C^*$), the $q$-difference operator $x\mapsto qx$ ($q\in\C^*$ not a root of unity), and the Mahler operator $x\mapsto x^p$ ($p\geq 2$ integer). The only restriction is that we constrain our solutions to be expressed as (possibly ramified) Laurent series in the variable $x$ (or in the variable $1/x$ in some special case associated with the shift operator). For a precise statement, we refer the reader to our main result, Theorem \ref{thm: main}, which  confirms guess (A) 
for these three operators.

\subsection{Statement of our main result}\label{sec:mainresult}

All along this paper, our framework consists in a tower of fields extensions 
$\C\subset K \subset F_0 \subset F$ 
with the following properties.  
\begin{itemize}
\item The field $K$ is equipped with an automorphism $\rho$ and a derivation 
$\tr$\footnote{A derivation $\tr$ on $K$ is a map from $K$ into itself satisfying $\tr(a+b)=\tr(a)+\tr(b)$ and $\tr(ab)=\tr(a)b+a\tr(b)$ for every $a,b\in K$.}. 
\item The automorphism $\rho$ extends to $F_0$ and $F$.
\item  The derivation $\tr$ extends to $F$, but $F_0$ is not necessarily closed under $\tr$.
\end{itemize}
Our aim is to prove that, if $f\in F_0$ is solution to a linear $\rho$-equation with coefficients in $K$, then 
either $f\in K$ or $f$ is hypertranscendental over $K$.   
Specifically, we consider the following four situations, which we refer to as Cases 
$\textbf{S}_{0}$, $\textbf{S}_{\infty}$, $\textbf{Q}$, and $\textbf{M}$, respectively.

\medskip

\noindent {\bf Case} $\textbf{S}_{0}$. In this case, we let $K=\C(x)$, $\tr= \frac{d}{dx}$, and
$\rho$ denote the automorphism of $K$ 
defined by  
$$
\rho(f(x))=f(x+h), \, \quad h \in \C^* . 
$$ 
The automorphism $\rho$ and the derivation $\tr$ naturally extend to  the  
of field $F:=\mathcal{M}er(\C)$ of meromorphic functions on $\C$.   
Let ${K_0:=\C(x,\exp (*x))}$ denote the field generated over $\C(x)$ by the functions $\exp (\ell x)$, $\ell\in \C$. 
In this setting, we let $F_0$ be any field extension of $\C(x)$
on which $\rho$ extends and satisfying the  following three conditions. 
 \begin{itemize}
 
\item[(i)] $F_0\subset  \mathcal{M}er(\C)$.

\item[(ii)] $\left\{f\in F_0 \mid \rho(f)=f\right\}=\C$.

\item[(iii)] $F_0\cap K_0=\C(x)$.
\end{itemize}

\begin{rem} For $h=1$, let us see that we can choose $F_{0}=\C(x,\Gamma(x))$, where $\Gamma (x)$ is the 
Gamma function. The functional equation $\Gamma(x+1)=x\Gamma(x)$ shows that $\rho$ extends to $F_0$ and it is 
well-known that
$\Gamma\in\mathcal{M}er(\C)$. Thus, (i) holds.  Now,  let $\beta\in F_0$ be such that $\rho(\beta)=\beta$. 
Writing $\beta$ as a rational function in $\Gamma$ with coefficients in $\C(x)$, and then using the functional equation  satisfied by $\Gamma$ 
and the fact that $\Gamma$ is transcendental over $\C(x)$, we can deduce that $\beta\in\C$. Hence (ii) holds. 
Finally, to prove (iii) it is sufficient to show that $\Gamma$ is transcendental over $K_{0}$. The latter property follows from classical asymptotics showing that, for every $\alpha (x)\in K_{0}$, we have 
$\lim_{x\to +\infty} |\Gamma (x)/\alpha (x)|=+\infty$.  
Thus, Case $\textbf{S}_{0}$ of Theorem \ref{thm: main} is a generalization of H\"older's theorem.
\end{rem}

\medskip

\noindent {\bf Case} $\textbf{S}_{\infty}$. In this case, we consider $K$, $\rho$, and $\tr$ 
as in Case $\textbf{S}_0$.  
We note that the automorphism $\rho$ and the derivation $\tr$ naturally extend to  the  
of field of Laurent series $\C((x^{-1}))$, and we set  
$F_0=F=\C((x^{-1}))$.

\medskip

\noindent  {\bf Case} $\textbf{Q}$.  In this case, we let $K=\displaystyle\bigcup_{j\geq 1} \C(x^{1/j})$ denote the field of ramified rational functions. We also use the notation $\C(x^{1/*})$ for this field. 
Given $q\in \C^{*}$ that is not a root of unity, we let $\rho$ denote the automorphism of $K$ defined by 
$$
{\rho(f(x))=f(qx)},\, \quad \tr= x\frac{d}{dx},
$$
and we let $F_0=F=\displaystyle\bigcup_{j\geq 1} \C((x^{1/j}))$ 
denote the field of Puiseux series. We also use the notation $\C((x^{1/*}))$ for this field.

\medskip

\noindent  {\bf Case} $\textbf{M}$.  In this case, we let $K=\C(x^{1/*})$ and, given a natural number $p\geq 2$, 
we let $\rho$ denote the automorphism of $K$ defined by 
$$
\rho(f(x))=f(x^p)\,,
$$ 
and we set $\tr= x\frac{d}{dx}$ and $F_0=F=\C((x^{1/*}))$.

\medskip

In all cases, $\tr$ is  a derivation on $K$ and $F$. An element $f$ in $F$ is 
said to be \emph{differentially algebraic} or $\tr$-algebraic over $K$ if there exists 
$n\in \N$ such that  the functions $f,\tr(f),\dots,\tr^{(n)}(f)$ are algebraically dependent over $K$.
 Otherwise, $f$ is said to be 
\emph{hypertranscendental}  over $K$. 
  Our main result reads as follow.

\begin{theo}\label{thm: main}
Let $K$, $F_0$, and $\rho$ be defined as in Cases $\textbf{S}_{0}$, $\textbf{S}_{\infty}$, \textbf{Q}, and \textbf{M}, 
and let $n$ be a positive integer. 
Let $f \in F_0$ be solution to the $\rho$-linear difference equation of order $n$ 
\begin{equation}\label{eq: difference}
 \rho^{n} (y) + a_{n-1}\rho^{n-1}(y)+\cdots + a_{1} \rho(y)+a_{0} y=0\,,
\end{equation}
where $a_{0},\ldots,a_{n-1}\in K$.   
Then either $f\in K$ or $f$ is hypertranscendental over $K$. 
\end{theo}

On the road to Theorem \ref{thm: main}, there is already a whole bunch of partial results and 
a number of them are used in our proof. These results can roughly be divided into three different types.

\begin{itemize}
\item Those concerned with affine equations of order one, that is with equations of the form $\rho(y)=ay+b$. 
References include \cite{Ho,Mo,Ma30,Ni84,rande1992equations,ishizaki1998hypertranscendency,Ha08,HS08,NG12}.

\item Those concerned with equations whose difference Galois group is large (\emph{e.g.}, simple, semisimple, reductive). 
References include \cite{HS08,DHR18,DHR2,AS17,ArDR18}. 
In particular, such results have nice applications to equations of order $2$. 

\item  Those dealing with general equations, but reaching only nonholonomicity instead of hypertranscendence 
(see Theorem \ref{lem: linlin} and Remark \ref{rem:linlin} below).  
References include  \cite{ramis1992growth,bezivin1993solutions,bezivin1994classe,SS16}.
\end{itemize}

Among these results, we quote the following one, which  
may be considered as a first step towards a proof of Theorem \ref{thm: main} (see Remark \ref{rem:linlin} for precise attribution).

\begin{theo}\label{lem: linlin}
In each of Cases $\textbf{S}_{0}$, $\textbf{S}_{\infty}$, $\textbf{Q}$, and $\textbf{M}$, the following property holds.  
Any element of $F_0$ that satisfies both a $\rho$-linear  equation and a 
$\tr$-linear equation with coefficients in $K$ belongs to $K$. 
\end{theo}

\begin{rem}\label{rem:linlin} Theorem \ref{lem: linlin} can be rephrased by saying that, if $f$ is as in 
Theorem \ref{thm: main}, then either $f\in K$ or $f$ is not holonomic.  
Note that recently the authors of \cite{SS16} give a uniform proof for all cases of Theorem \ref{lem: linlin}, 
see Corollaries 3 and 5 therein. 
In each case, we also mention the original reference. 
Case $\textbf{S}_{0}$ is due to B\'ezivin and Gramain \cite{bezivin1993solutions}.  Case $\textbf{S}_{\infty}$ is due to Sch\"afke and Singer 
\cite[Corollary~5]{SS16}. 
Furthermore, in Cases $\textbf{Q}$ and $\textbf{M}$, a change of variable 
of the form $z=x^{1/\ell}$, $\ell\in \N^*$, can be used to  reduce the situation to the case where $K=\C(x)$ and $F_{0}=\C((x))$. In the latter situation,    
Case $\textbf{Q}$ is due to Ramis 
\cite{ramis1992growth}, while Case $\textbf{M}$ is due to B\'ezivin \cite{bezivin1994classe}.
\end{rem}
\subsection{Strategy of proof } 

Elements of the rings Rat, Alg, and Hol are well-understood. We have recurrence formulas for their coefficients, 
precise asymptotics, and a good understanding of their finitely many singularities.  
In contrast, our knowledge about differentially algebraic functions is much more 
limited (see, for instance, the survey \cite{Ru}). In fact, such functions can behave quite wildly. For example, 
the function $\sum x^{n^2}$ is differentially algebraic and admits the unit circle as a natural boundary\footnote{Incidentally, this example shows that the "outrageous conjecture" suggested at the end of \cite{Ru} is false. 
Indeed, it follows from a theorem of Nesterenko that $\sum 2^{-n^2}$ is transcendental. }. 
Thus, proving hypertranscendence required more involved tools, and is substantially more difficult than proving irrationality, transcendence, or nonholonomicity. In this direction,  
Hardouin and Singer \cite{HS08} built a new Galois theory of difference equations, which is  
designed to measure the differential dependencies among solutions to linear difference equations.  
In particular, 
it applies to hypertranscendence. We refer the reader to \cite{DV} and \cite{Ha16} for an 
introduction to this topic.  
Our proof of Theorem \ref{thm: main} remains on this theory. 

Let us briefly describe our strategy.  
Since $f$ satisfies a linear difference equation, it is classical to associate with this equation 
its difference Galois group, which is a linear 
algebraic group that encodes the algebraic relations between the solutions to this equation. 
The more involved parametrized Galois theory developed in \cite{HS08} 
attaches to any linear difference equation, a geometric object, 
\emph{the parametrized Galois group}, whose structure encodes the algebraic relations satisfied by the solutions and their derivatives. 
This is not a linear algebraic group anymore, but a linear differential group (see Section \ref{subsec: PDGT}).  
Nevertheless, both groups are strongly related and, viewed on a suitable field extension, the latter is Zariski dense in the former (see 
Proposition \ref{prop: paradense}). 
In this framework, we can conclude that $f$ is hypertranscendental when the corresponding parametrized Galois group  
is \emph{large}. Furthermore, there are several results, culminating in \cite{AS17}, 
showing that if the classical difference Galois group is large, then both groups are the same when 
viewed on a suitable field extension.  We emphasize that such results follow a  strategy initiated in \cite{HS08} and developed further in \cite{DHR18}.  It combines  a fundamental result about classification  of 
 differential algebraic subgroups  of semisimple algebraic groups  by Cassidy  \cite{Ca72,Ca89}, parametrized Galois correspondence,  and Theorem \ref{lem: linlin}.

We prove Theorem \ref{thm: main} by induction on the order $n$ of the $\rho$-linear equation satisfied by 
$f$. We first prove the result for affine equation of order one. Then we show that, iterating the associated linear system 
if necessary, one can reduce the situation to the case where the associated difference Galois group $G$ is connected. 
Now, if $G$ acts irreducibly on $\C^n$, we show how to reduce the situation to the case where $G$ is semisimple. 
In that case, the difference Galois group is large enough and a recent result of Arreche and Singer \cite{AS17} allows 
us to conclude that $f$ is hypertranscendental.  
Finally, if $G$ acts reducibly on $\C^n$, we assume that $f$ is differentially algebraic and we show how to construct from $f$ a differentially algebraic function $g$  that satisfies a linear equation of smaller order. 
Furthermore, this construction ensures that $g\in K$ 
if and only if  $f\in K$. 
Applying the induction assumption to $g$, we deduce that $f\in K$, as wanted. Thus, once the case of order one 
equation is solved, the proof really makes a dichotomy between \emph{irreducible} and \emph{reducible} difference Galois groups.

\subsection{An application to number theory} 

At the end of the 1920s, Mahler \cite{Ma29,Ma30a,Ma30} introduced a new method for proving transcendence 
and algebraic independence of values at algebraic points of analytic functions satisfying different types of 
functional equations associated with the operator $x\mapsto x^p$. Mahler's method aims 
at transferring results about algebraic independence 
over $\overline{\mathbb Q}(x)$ of such functions  
 to the algebraic independence over $\overline{\mathbb Q}$ of their values at algebraic points. 
 A major result in this direction is Nishioka's theorem \cite{Ni}, which provides an analogue of the 
 Siegel-Shidlovskii theorem for linear Mahler systems.  
Let us also mention that significant progress in this theory has been made recently \cite{PPH,AF17,AF1,AF2}. 
Combining Case $\textbf{M}$ of Theorem \ref{thm: main} with Nishioka's theorem and following the line of argument 
 given in 
the proof of Proposition 4.1 in \cite{AF2}, we can deduce the following general result. 
We recall that $f(x)\in \overline{\mathbb Q}[[x]]$ is a $p$-\emph{Mahler function} 
if there exist  polynomials $a_0(x),\ldots,a_n(x)\in\overline{\mathbb Q}[x]$, not all zero, such that
$$
a_0(x)f(x)+a_1(x)f(x^p)+\cdots+a_n(x)f(x^{p^n})=0\,.
$$
We simply say that $f$ is a Mahler function if it is $p$-Mahler for some 
integer $p\geq 2$. 
According to Theorem \ref{thm: main}, a Mahler function is thus either rational or hypertranscendental. 
We recall that an irrational Mahler function is meromorphic on the open unit disc and 
admits the unit circle as a natural boundary. 

\begin{theo}\label{thm: app}
Let $f(x) \in \overline\Q[[z]]$ be a Mahler function that is not rational, let $r$ be a positive integer, 
and let $\mathcal K$ be a compact 
subset of the open unit disc. Then, for all but finitely many algebraic numbers 
$\alpha\in\mathcal K$, the complex numbers 
$f(\alpha),f'(\alpha),\ldots,f^{(r)}(\alpha)$ are algebraically independent over $\overline\Q$.  
\end{theo}

Results of this type have been first obtained by Mahler 
\cite{Ma30} for solutions to affine order one equations. 
The main feature of Theorem \ref{thm: app} is that it applies to \emph{any} irrational Mahler function. 
It is also almost best possible in the sense that one cannot avoid the possibility of finding infinitely many exceptions in the 
whole open unit disc. For instance, the transcendental $2$-Mahler function 
$$
f(z)=\prod_{n=0}^{\infty}(1-3z^{2^n})
$$
vanishes at all algebraic numbers $\alpha$ such that  $\alpha^{2^n}=1/3$ for some integer $n\geq 0$. 
This shows that, even for $r=1$, the exceptional set can depend on $\mathcal K$.   However, if $\mathcal K$ is fixed, we do not know whether 
the exceptional set always remains finite when $r$ tends to infinity.

\subsection{Organization of the paper} 
This article is organized as follows. In Section \ref{sec: galois}, we provide a short  
introduction to difference Galois theory and to parametrized difference Galois theory, following 
\cite{vdPS97,HS08}.  In Section~\ref{sec: paramframework}, we describe, for each of Cases   
$\textbf{S}_{0}$, $\textbf{S}_{\infty}$, $\textbf{Q}$, and $\textbf{M}$, the suitable framework 
that is needed to use the 
parametrized difference Galois theory of \cite{HS08}. Several auxiliary results are gathered in Section 
\ref{sec: aux}. They concern algebraic groups, the behavior of the difference Galois group when considered 
over different field extensions, 
and Picard-Vessiot extensions associated with iterated systems. 
Finally, Section \ref{sec: main} is devoted to the proof of Theorem~\ref{thm: main}. 

\section{Difference and parametrized difference Galois theories}\label{sec: galois}

In this section, we provide a short introduction to difference Galois theory and to parametrized difference 
Galois theory. 

\subsection{Difference Galois theory}\label{subsec: DGT}

We first recall some notation, as well as classical results, concerning  difference Galois theory. 
We refer the reader to \cite{vdPS97} for more details. 

\subsubsection{Notation in operator algebra}

A difference ring is a pair $(R,\rho)$ where $R$ is a ring and $\rho$ is a ring automorphism of $R$. 
If $R$ is a field, then $(R,\rho)$ is called a difference field or a $\rho$-field. 
An ideal $I$ of $R$ such that $\rho(I)\subset I$ is called a difference ideal or a $\rho$-ideal. 
We say that the difference ring $(R,\rho)$ is simple if the only $\rho$-ideals of $R$ are $\{0\}$ and $R$. 
Two difference rings $(R_1,\rho_1)$ and $(R_2,\rho_2)$ are isomorphic 
if there exists a ring isomorphism $\varphi$ 
between $R_1$ and $R_2$ such that $\varphi\circ \rho_1=\rho_2\circ \varphi$.  
A difference ring $(S,\rho')$ is a difference ring extension of $(R,\rho)$ if $S$ is a ring extension of $R$ 
and if $\rho'_{\mid R}=\rho$. In this case, we usually keep on denoting $\rho'$ by $\rho$.  
When $R$ is a $\rho$-field, we say that $S$ is a $R$-$\rho$-algebra. 
Two difference ring extensions $(R_1,\rho)$ and $(R_2,\rho)$ of the difference ring $(R,\rho)$ 
are isomorphic over $(R,\rho)$ if there exists a difference ring isomorphism $\varphi$ 
from $(R_1,\rho)$ to $(R_2,\rho)$ such that $\varphi_{\mid R}=\mbox{Id}_{R}$.  
The ring of constants of the difference ring $(R,\rho)$ is defined by 
$$
R^{\rho}:=\{f\in R \mid \rho(f)=f\} \,.
$$
If $R^{\rho}$ is a field, it is called the field of constants. 
If there is no risk of confusion, we usually simply say that $R$, instead of $(R,\rho)$, 
is a difference ring (or a difference field, or a difference ring extension...).  

\subsubsection{Difference equations and linear difference systems}  
A linear $\rho$-equation  of order $n$ over $K$ is an equation of the form 
\begin{equation}\label{eq: rho}
 \cL(y):=\rho^{n} (y) +a_{n-1}\rho^{n-1}(y)+ \cdots + a_{0} y=0 \,,
\end{equation}
with  $a_{0},...,a_{n-1} \in K$. If   $a_0 \neq 0$, 
this relation  can be written in matrix form as 
\begin{equation}\label{eq: companion}
\rho Y=A_\cL Y\,
\end{equation}
where
$$
 A_\cL :=\begin{pmatrix}
0&1&0&\cdots&0\\
0&0&1&\ddots&\vdots\\
\vdots&\vdots&\ddots&\ddots&0\\
0&0&\cdots&0&1\\
-a_{0}& -a_{1}&\cdots & \cdots & -a_{n-1}
\end{pmatrix} \in \mathrm{GL}_{n}(K) \,.
$$
The matrix $A_\cL$ is called the companion matrix associated with Equation~\eqref{eq: rho}. 
It is often more convenient to use the notion of linear difference system,  
that is of system of the form 
\begin{equation}\label{eq:systeminitial}
\rho(Y)=AY, \hbox{ with } A \in \GL_n(K).\,
\end{equation}

We recall that two difference systems $\rho (Y)=AY$ and $\rho (Y)=BY$ with $A,B \in \GL_{n}(K)$ are 
said to be \emph{equivalent over $K$} if  there exists a gauge transformation $T \in \GL_{n}(K)$ such that 
$B=\rho(T) AT^{-1}$. In that case, ${\rho (Y)=AY}$ if and only if $\rho (TY)=B (TY)$. 

\begin{rem}\label{rem:orderofthesystorderoftheequation}
By \cite[Theorem B.2.]{HenSin}, if $K$ contains a nonperiodic element with respect to $\rho$, then 
the cyclic vector lemma holds, and any linear difference system is equivalent to one given by a 
companion matrix associated with some linear equation. 
Let $L$ be a $\rho$-field extension of $K$ and let  $(f_0,f_1,\dots,f_{n-1})^\top \in L^n$ be a nonzero  
solution to a linear system $\rho(Y)=AY$, with $A \in \GL_n(K)$. Then each coordinate $f_j$ satisfies some 
linear $\rho$-equation over $K$ of order at most $n$. 
\end{rem}

\subsubsection{Picard-Vessiot rings and difference Galois groups}\label{sec:PVring}

A \emph{Picard-Vessiot ring} for \eqref{eq:systeminitial} over  a difference field  $(K, \rho)$ of characteristic zero 
is a $K$-$\rho$-algebra satisfying the following three properties. 

\smallskip
\begin{itemize}
\item[(1)] There exists $U \in \GL_{n}(R)$ such that $\rho(U)=AU$. Such a matrix $U$ is called a
\emph{fundamental matrix}.  

\smallskip

\item[(2)] $R$ is generated as a ring over $K$ by the coordinates of $U$ and by $\det(U)^{-1}$,  that is 
$R=K[U,\det(U)^{-1}]$. 

\smallskip

\item[(3)] $R$ is a simple difference ring.

\end{itemize}
 A Picard-Vessiot ring is not always an integral domain. However, it is a direct 
 sum of integral domains which are transitively permuted by $\rho$. More precisely, by \cite[Corollary 1.16]{vdPS97}, 
 there exist $r\geq 1$ and orthogonal idempotent elements $e_{0},\dots, e_{r-1}$, such that 
 $$
 R=e_{0}R\oplus \dots \oplus e_{r-1}R\,,
 $$ 
 where $e_{i}R$ is an integral domain and $\rho(e_{i})=e_{i+1 \mathrm{mod} r}$. We recall that 
Picard-Vessiot rings always exist. Indeed, it is sufficient to endow the polynomial $K$-algebra $K[X,\frac{1}{\det(X)}]$ with a structure of $K$-$\rho$-algebra by setting ${\rho(X)=AX}$. Then, for  any  maximal $\rho$-ideal  $\mathfrak{M}$  of 
$K[X,\frac{1}{\det(X)}]$, the quotient $R=K[X,\frac{1}{\det(X)}]/\mathfrak{M}$ is a Picard-Vessiot ring.  
However, this construction does not guarantee that  the ring of $\rho$-constant has not grown. This justifies the introduction of the   more convenient  notion of Picard-Vessiot extension.  

\medskip

A \emph{Picard-Vessiot extension} $\cQ$  for \eqref{eq:systeminitial} over  a difference field  $(K, \rho)$ of characteristic zero 
is a $K$-$\rho$-algebra  $\cQ$ satisfying the following properties. 

\smallskip
\begin{itemize}
\item[(1)] There exists $U \in \GL_{n}(\cQ)$ such that $\rho(U)=AU$. Such a matrix $U$ is called a
fundamental matrix.  

\smallskip

\item[(2)]  $\cQ$ is a \emph{pseudofield} extension of $K$, that is there exist a positive integer $r$, orthogonal idempotent elements $e_{0},\dots, e_{r-1}$ of $\cQ$, and 
 a field extension $L$ of $K$ such that $\cQ=e_0 L  \oplus \hdots \oplus e_{r-1} L$ and 
$\rho(e_i) = e_{i+1 \mathrm{mod} r}$. 

\smallskip 

\item[(3)] $\cQ$ is the smallest pseudofield extension of $K$ containing $U$. 

\smallskip

\item[(4)] $\cQ^{\rho}=K^{\rho}$.  

\end{itemize}

When  $\cQ$ is a field, we say that $\cQ$ is a \emph{Picard-Vessiot field extension}. 
By  \cite[$\S$1.1]{vdPS97}, if $K^{\rho}$ is algebraically closed, then  there exists a Picard-Vessiot extension and two Picard-Vessiot extensions are isomorphic as $K$-$\rho$-algebra. 
The relation between Picard-Vessiot-rings and Picard-Vessiot-extensions is given by the following proposition, which is a straightforward consequence of \cite[Proposition 2.5 and Corollary 2.6]{OW15}.

\begin{propo}\label{prop:PvringPvext}
If $K^{\rho}$ is an algebraically closed field of characteristic zero, then the following properties hold. 

\begin{itemize}
\item Let $\cQ$ be a Picard-Vessiot extension over $K$ and let us define ${R:=K[U, \frac{1}{\det(U)}] \subset \cQ}$, 
where $U$ is a fundamental matrix. 
Then $R$ is a Picard-Vessiot ring.
\item If $R$ is a Picard-Vessiot ring over $K$ then the total quotient ring\footnote{We recall that, 
given a ring $R$, its total quotient ring is defined as the localization of $R$ at the multiplicative set 
formed by all nonzero divisors of $R$.} 
of $R$ is a Picard-Vessiot extension.  
\end{itemize}
\end{propo}

From now on, we assume that $K$ is a $\rho$-field such that $\mathbf{k}:=K^\rho$ 
is an algebraically closed field of 
characteristic zero. 
Let $\cQ$ be a Picard-Vessiot extension over $K$. The  \emph{difference Galois group} $\Gal(\mathcal{Q}/K)$ of \eqref{eq:systeminitial} over $K$ 
is defined as the group of $K$-$\rho$-algebra automorphisms of $\mathcal{Q}$:
$$
\Gal(\mathcal{Q}/K) :=\{ \sigma \in \mathrm{Aut}(\mathcal{Q}/K) \ | \ \rho\circ\sigma=\sigma\circ \rho \}.
$$

For any fundamental matrix $U \in \GL_n(\cQ)$, an easy computation shows that 
$U^{-1}\sigma(U) \in \GL_{n}(\mathbf{k})$ for all $\sigma \in \Gal(\mathcal{Q}/K)$. 
By  \cite[Theorem~1.13]{vdPS97}, the faithful representation
\begin{eqnarray*}
 \Gal(\mathcal{Q}/K) & \rightarrow & \GL_{n}(\mathbf{k}) \\ 
  \sigma & \mapsto & U^{-1}\sigma(U)
\end{eqnarray*}
identifies $\Gal(\mathcal{Q}/K) $ with a linear algebraic subgroup $G\subset\GL_{n}(\mathbf{k})$. 
Choosing another fundamental matrix of solutions $U$ leads to a conjugate representation. 

\subsubsection{Torsor and algebraic relations}  
A fundamental result in difference Galois theory (\cite[Theorem 1.13]{vdPS97}) says that 
the Picard-Vessiot ring $R$ is the coordinate ring of a $\Gal(\cQ/K)$-torsor over $K$. 
Thereby, the difference Galois group controls the algebraic
relations satisfied by the solutions to the underlying linear system.  As a corollary
of this structure of $\Gal(\cQ/K)$-torsor, one obtains the fundamental equality
$$
\dim_{\mathbf{k}} \Gal(\cQ/K)= \trd_K \cQ \,,
$$
where we let $\trd_K \cQ$ denote the transcendence degree of the field extension $L/K$. Here,  
the field $L$ is defined as in (2) of the definition of a Picard-Vessiot extension.


\subsubsection{The Galois correspondence}  

The Galois correspondence for linear difference systems can be summarized as 
follows (see \cite[Theorem 1.29]{vdPS97}).  

\begin{theo}\label{thm: correspondence}
Let $\cQ$ be a Picard-Vessiot extension over $K$. Let $\mathcal R$ denote the set of  $K$-$\rho$-algebras $F$ such that $F\subset \mathcal Q$ and such that 
every nonzero divisor of $F$ is a unit of $F$. Let $\mathcal G$ denote the set of algebraic subgroups of 
$\Gal(\mathcal Q/K)$. Then the following properties hold. 
\begin{itemize}
\item[-] For any $F\in \mathcal R$, the set $$G(\mathcal Q/F):=\{\sigma \in \Gal(\mathcal Q/K) \mid \sigma(f)=f, \;\;\;\forall f\in F\}$$ 
is an algebraic subgroup of $\Gal(\mathcal Q/K)$. 

\item[-] For any $H\in \mathcal G$, the set $$\mathcal Q^{H}:=\{f\in \mathcal Q  \mid \sigma(f)=f, \;\;\;\forall \sigma\in H\}$$ 
belongs to $\mathcal R$. 

\item[-] The maps 
$$\begin{array}{ll}r: &\mathcal R \to \mathcal G\\
&F\mapsto G(\mathcal Q/F)\end{array} \;\;\mbox{ and }\;\; \begin{array}{ll}g:&\mathcal{G} \to \mathcal R\\&
H\mapsto \mathcal Q^{H}\end{array}$$ 
are each other's inverses. 
\end{itemize}
\end{theo}

\begin{rem}
In the case where $R$ is an integral domain, $\mathcal Q$ is a difference field, and Theorem \ref{thm: correspondence} provides a correspondence 
between the difference subfields of $\mathcal Q$ containing $K$, on the one hand, and the algebraic subgroups of 
$\Gal(\mathcal Q/K)$, on the other hand. 
\end{rem}

\subsection{Parametrized difference Galois theory}\label{subsec: PDGT}

We recall here some notation and results concerning parametrized difference Galois theory. 
We refer the reader to \cite{HS08} for more details. 

\subsubsection{Notation in operator algebra} 
A differential ring is a pair $(R,\delta)$ where $R$ is a ring  
and $\delta$ is a derivation on $R$, that is 
a homomorphism of additive group 
from $R$ into itself satisfying the Leibniz rule:   
$$
\delta(ab)=\delta(a)b+a\delta(b) \,.
$$ 
If $R$ is a field, then $(R,\delta)$ is called a differential field or a $\delta$-field.  
The ring of constants  of $R$ is defined by 
$$
R^\delta =\{ f \in R \mid  \delta(f)= 0\}\,.
$$
 Let $L$ be a $\delta$-ring extension of a $\delta$-ring $R$. Given a subset $S$ of $L$, we let 
  $R\{S\}_{\delta}$ denote the smallest  $R$-$\delta$-algebra generated by $S$. 
A field $K$ endowed with a structure of both difference and differential field is called 
a $(\rho,\delta)$-field if $\rho$ and $\delta$ \emph{ commute}, that is if 
$$
\rho\circ\delta (f)= \delta\circ\rho (f)
$$ 
for every $f\in K$.

The study of algebraic structures of difference and differential fields is the main object of the so-called 
differential and difference algebra.  We refer the interested reader to the founding books of Kolchin 
(\cite{Ko73}) for further details concerning 
differential algebra and to the book of Cohn for further details concerning difference algebra (\cite{Cohn}). 
It is worth mentioning that these two setting are very different. 
In order to avoid many analogous definitions, we use the following convention:  we add  an 
 operator prefix to an algebraic attribute to signify the compatibility of the algebraic structure with respect to the operator. 
 For instance, a $(\rho,\delta)$-ring is a ring equipped with $\rho$ and $\delta$,  a $(\rho,\delta)$-field extension 
 $K \subset L$ is a field extension such that the fields $L$ and $K$ are $(\rho,\delta)$-fields and such that the action 
 of $\rho$ and  $\delta$ are compatible with the field extension, and so on.

 \subsubsection{$(\rho,\delta)$-Picard-Vessiot ring and the parametrized difference Galois group}
 \label{sec:PPV}

 Let  $\tK$ be a $(\rho,\delta)$-field and let $A \in \GL_n(\tK)$. We consider the difference system 
\begin{equation}\label{eq:eqinittK}
\rho(Y)=AY.
\end{equation}

 A $(\rho,\delta)$-\emph{Picard-Vessiot ring} for \eqref{eq:eqinittK} over $\tK$ is a $(\rho,\delta)$-ring extension 
 $\widetilde{R}$ of $\tK$ satisfying the following three properties. 

\medskip 

\begin{itemize}
\item[(1)] There exists $U \in \GL_{n}(\widetilde{R})$ such that $\rho (U)=AU$. 
Such a matrix $U$ is called a fundamental matrix. 

\smallskip

\item[(2)] $\widetilde{R}$ is generated as a $\delta$-ring over $\tK$  
by the coordinates of $U$ and by $\det(U)^{-1}$, that is 
$\widetilde{R}=\tK\{U,\det(U)^{-1}\}_\delta$.

\smallskip

\item[(3)] $\widetilde{R}$ is a simple $(\rho,\delta)$-ring, that is the only $(\rho,\delta)$-ideals of $\widetilde{R}$ 
are $\{0\}$ and $\widetilde{R}$.
\end{itemize}

\medskip

A $(\rho,\delta)$-Picard-Vessiot ring is not always an integral domain but it is a direct sum of integral domains 
closed under  $\delta$ and transitively permuted by $\rho$. A construction similar to that of Section \ref{sec:PVring} 
shows that $(\rho,\delta)$-Picard-Vessiot rings always exist. Again, this construction does not ensure that 
the ring of $\rho$-constants has not grown. This justifies the introduction of the  more convenient  notion of 
$(\rho,\delta)$-Picard-Vessiot extension. 

A $(\rho,\delta)$-\emph{Picard-Vessiot extension} $\tcQ$  for \eqref{eq:eqinittK} over 
$\tK$ is a $\tK$-$\rho$-$\delta$-algebra 
satisfying the following  properties.  

\smallskip
\begin{itemize}
\item[(1)] There exists $U \in \GL_{n}(\tcQ)$ such that $\rho(U)=AU$. Such a matrix $U$ is called a
fundamental matrix.   

\smallskip

\item[(2)] $\tcQ$ is a pseudo $\delta$-field extension of $\widetilde{K}$, that is there exist 
a positive integer $r$, orthogonal idempotent elements $e_{0},\dots, e_{r-1}$ of $\tcQ$, 
and $\tL$ a $\delta$-field extension of $\widetilde{K}$ such that $\tcQ=e_0 \tL  \oplus \hdots \oplus e_{r-1} \tL$ 
and $\rho(e_i) = e_{i+1 \mathrm{mod} r}$.

\smallskip 

\item[(3)] $\tcQ$ is equal to $\tK\langle U \rangle$, that is the smallest pseudo $\delta$-field extension 
of $\tK$ containing $U$. 

\smallskip

\item[(4)] $\tcQ^{\rho}=\tK^{\rho}$.

\end{itemize}

When $\tcQ$ is a field, we say that $\tcQ$ is a 
$(\rho,\delta)$-\emph{Picard-Vessiot field extension}. By  \cite[Corollary 9]{Wibmerexistence}, 
if $\tK^{\rho}$ is algebraically closed, then  there exists a $(\rho,\delta)$-Picard-Vessiot extension $\tcQ$ for \eqref{eq:eqinittK} over $\tK$. The smallest $\tK$-$(\rho,\delta)$-subalgebra  $\tK\{U,\frac{1}{\det(U)}\}_{\delta}$ 
of $\tcQ$ generated  by a fundamental matrix and the inverse  of its determinant  is a $(\rho, \delta)$-Picard-Vessiot ring. However, in order to ensure the uniqueness of  $(\rho, \delta)$-Picard-Vessiot rings 
 up to $\tK$-$(\rho,\delta)$-algebra isomorphisms, one needs the field $\tK^\rho$ to be $\delta$-closed (\cite[Proposition~6.16]{HS08}).
 We recall that a  differential field $(L,\delta)$ is called \emph{differentially closed} 
or $\delta$-closed if, for every set of $\delta$-polynomials $\mathcal F$,  
the system of $\delta$-equations $\mathcal F=0$ has a solution in some $\delta$-field extension of $L$
 if and only if it has a solution in $L$. Note that a $\delta$-closed field is algebraically closed. 
Differentially closed fields are huge and, for instance, none of the function fields $K$ introduced in  Section  \ref{sec:mainresult} satisfies the assumption that $K^\rho$ is  differentially closed. Nonetheless,  working with $\delta$-closed fields allows us to simplify  many arguments. Thus, we will embed  $K$ into a $(\rho, \delta)$-field $\tK$ whose 
field of $\rho$-constants is $\delta$-closed, and we will use  some descent argument to go back to $K$ if necessary (see Sections \ref{sec: paramframework} and \ref{sec:goingupdown}).

From now on, we let $\tK$ denote a $(\rho,\delta)$-field whose field of $\rho$-constants $\widetilde{C}:=\tK^\rho$ 
is $\delta$-closed. Let $\tcQ$ be a $(\rho,\delta)$-Picard-Vessiot extension over $\tK$. 
The \emph{parametrized difference Galois group} 
$\Gal^{\delta}(\widetilde{\mathcal{Q}}/\tK)$ of \eqref{eq:eqinittK} over $(\tK,\rho,\delta)$, also called the 
$(\rho,\delta)$-Galois group,  
is defined as the group of 
$\tK$-$(\rho,\delta)$-algebra automorphisms of $\tcQ$: 
$$
\Gal^{\delta}(\widetilde{\mathcal{Q}}/\tK):=
\left\{ \sigma \in \mathrm{Aut}(\widetilde{\mathcal{Q}}/\tK) \Big| 
\ \rho\circ\sigma=\sigma\circ \rho \text{ and } \delta\circ\sigma=\sigma\circ \delta \right\} \,. 
$$
The parametrized difference Galois group $\Gal^\delta(\widetilde{\mathcal{Q}}/\tK)$  
 is a geometric object that encodes the differential algebraic relations between the solutions to \eqref{eq:eqinittK}. See \cite[Proposition 6.24]{HS08} for more details. 
Roughly speaking, the larger $\Gal^\delta(\widetilde{\mathcal{Q}}/\tK)$, the fewer $\delta$-algebraic 
relations over $\tK$ that hold among the elements of $\widetilde{\mathcal{Q}}$.    

For any fundamental  matrix $U \in \GL_n(\tcQ)$, an easy computation shows that  $U^{-1}\sigma(U)\in \GL_{n}(\tC)$ 
for any element ${\sigma \in \Gal^{\delta}(\widetilde{\mathcal{Q}}/\tK)}$.  
By \cite[Proposition~6.18]{HS08}, the faithful representation
\begin{eqnarray*}
 \Gal^\delta(\widetilde{\mathcal{Q}}/K) & \rightarrow & \GL_{n}(\tC) \\ 
 \sigma & \mapsto & U^{-1}\sigma(U)
\end{eqnarray*}
identifies  $\Gal^\delta(\widetilde{\mathcal{Q}}/\tK) $ with a \emph{linear differential algebraic subgroup}   
$H\subset\GL_{n}(\tC)$, i.e. a subgroup of $\GL_{n}(\tC)$ that 
can be defined as the vanishing set of polynomial $\delta$-equations with coefficients in $\tC$. 
Furthermore, choosing another fundamental matrix $U$ leads to a conjugate representation.  

 \subsubsection{Comparison between difference Galois groups and parametrized difference Galois groups} 

 Let $\widetilde{R}$ denote a $(\rho,\delta)$-Picard-Vessiot ring 
 for the system \eqref{eq:eqinittK}  over $\tK$, let 
$\widetilde{\mathcal{Q}}$ denote the associated  $(\rho,\delta)$-Picard-Vessiot extension, 
and let $U \in \GL_n(\widetilde{R})$ be a fundamental matrix of solutions. 
By \cite[Proposition 6.21]{HS08}, $R:= \tK[U,\det (U)^{-1}]$  is a Picard-Vessiot ring for \eqref{eq:eqinittK} over $\tK$. 
Let $\mathcal{Q}$ denote the  Picard-Vessiot extension
corresponding to $R$.  Then we have 
${\mathcal{Q} \subset \widetilde{\mathcal{Q}}}$ 
and one can identify $\Gal^{\delta}(\widetilde{\mathcal{Q}}/\tK)$ with a subgroup of 
$\Gal(\mathcal{Q}/\tK)$ by restricting the elements of  
$\Gal^{\delta}(\widetilde{\mathcal{Q}}/\tK)$ to $\mathcal{Q}$.  
The following result provides a fundamental link between classical difference Galois theory 
and its parametrized counterpart, see \cite[Proposition 6.21]{HS08}.

\begin{propo}\label{prop: paradense}
The parametrized difference Galois group $\Gal^{\delta}(\widetilde{\mathcal{Q}}/\tK)$ is a 
Zariski-dense subgroup of  $\Gal(\mathcal{Q}/\tK)$.
\end{propo}

 \section{The parametrized  framework }\label{sec: paramframework}

Let us go back to our $K,F_0,F,\rho,\tr$ corresponding to Cases $\textbf{S}_{0}$, $\textbf{S}_{\infty}$, \textbf{Q}, 
and \textbf{M} (see Section \ref{sec:mainresult}). 
We first observe that, in each case, $K^{\rho}=\C$ is algebraically closed and 
$K$ has characteristic zero. Hence the theory of \cite{vdPS97} described in Section \ref{subsec: DGT} applies. 
In this section, we describe some suitable field extensions $\widetilde{K}$ of $K$, and $\widetilde{C}$ of 
$C= F^{\rho}$,  
that allow us to apply the parametrized difference Galois theory of \cite{HS08} described in Section 
\ref{subsec: PDGT}.

We first recall that any differential field has a differential closure, {\it i.e.} a differentially closed field extension that is minimal with respect to this property.  In Case $\textbf{S}_0$, we set $\delta:=\tr$, and we let $\tC$ denote 
a $\delta$-closure of $C_h$, the subfield of $\cM er(\C)$ formed by those meromorphic functions  that are $h$-periodic.  
Note that $C_h$ is also a $\delta$-field for $\rho$ and $\delta$ commute.  Sometime, we will use the notation 
$\widetilde{C_{h}}$ instead of $\widetilde{C}$.  
In Cases $\textbf{S}_{\infty}$ and $\textbf{Q}$,  
we also set $\delta:=\tr$, and we let $\widetilde{\C}$ denote 
a $\delta$-closure of $\C$ and we set $\tC:=\widetilde{\C}$.   
By \cite[Lemma~2.3]{DHR18}, in these three cases, the field 
$$\widetilde{K}:=\mathrm{Frac}(K \otimes_{\C} \widetilde{C})$$ 
is a   $(\rho,\delta)$-field extension  of $K$ 
such that 
$\widetilde{K}^{\rho}=\widetilde{C}$.

\medskip
  
Case \textbf{M} is a bit more tricky. Indeed, in order to apply the theory of \cite{HS08}, we need to consider 
a derivation $\delta$ on $\widetilde{K}$ such that $\delta$ and $\rho$ commute.  
Following \cite{DHR18}, we let $\log$ denote a transcendental element over 
$K$, and we define $\delta=\log\times \tr$ to be the derivation that acts on $K(\log)$ by 
$$
\delta (\log)=\log  \mbox{ and }\delta (x^{a})=ax^a\log \,.
$$
Then we define a structure of $\rho$-field on $K(\log)$ by setting $\rho (\log)=p\log$. 
We observe that $K(\log)^{\rho}=\C$ and that $\rho$ and $\delta$ commute. 
Finally, we consider $\widetilde{\C}$ a 
$\d$-closure  of  $\C$ such that $\rho$ acts trivially on $\widetilde{\C}$.  
We set  $\tC:=\widetilde{\C}$ and by \cite[Lemma 2.3]{DHR18}, the field 
$$\widetilde{K}:=\mathrm{Frac}(K(\log) \otimes_{\C} \widetilde{C})\, ,$$ 
 is a $(\rho,\delta)$-field such that 
$\widetilde{K}^{\rho}=\widetilde{C}$.

In all cases, $K$ is a $\rho$-field of characteristic zero 
such that $K^{\rho}=\C$, $F$ is a $\rho$-field extension of $K$ with $C:=F^\rho$, and $\widetilde{K}$ is a 
$(\rho,\delta)$-field of characteristic zero, such that $\widetilde{K}^{\rho}=\widetilde{C}$, where $\widetilde{C}$ is differentially closed. Furthermore, $\rho$ and $\delta$ commute. 
The following table summaries the different frameworks  
we consider in this paper. 

$$
\begin{array}{|l|l|l|l|l|l|l|l|l|l}\hline 
\hbox{Case}&F & K & C   & \tr &\delta &\widetilde{C} &\widetilde{K} \\ \hline  
\mathbf{S}_{0}& \mathcal{M}er(\C)& \C(x)& C_h &\frac{d}{dx}&\frac{d}{dx}& \widetilde{C_h}&  \widetilde{C_h}(x) 
 \\   \hline  

\mathbf{S}_{\infty}&\C((x^{-1}))&  \C(x)& \C& \frac{d}{dx}&\frac{d}{dx}&   \widetilde{\C}& \widetilde{\C}(x)  \\ \hline  

 \mathbf{Q}&\C((x^{1/*}))  & \C(x^{1/*})& \C& x \frac{d}{dx}&x\frac{d}{dx} &  \widetilde{\C}
   & \widetilde{\C}(x^{1/*})\\ \hline  
 
\mathbf{M}&\C((x^{1/*}))&  \C(x^{1/*})& \C& x\frac{d}{dx} &\log \times x\frac{d}{dx}& \widetilde{\C} 
& \widetilde{\C}(x^{1/*})(\log)\\ \hline
\end{array}
$$

\medskip

\begin{rem}\label{rem:diffconstant}
Note that the field of $\delta$-constants of $C$ is the algebraically closed field $\C$. 
Then, by \cite[Lemma 9.3]{CS07}, the field of $\delta$-constants of $\tC$ is also $\C$. 
At some places, we will consider some iteration of the operator $\rho$. In Case $\bS_0$, 
it might happen that we start with the $\rho$-field extension $\tK=\widetilde{C_h}(x)$ of $K$, and then consider this extension as a $\rho^r$-field extension. Note that $\widetilde{C_h} \neq \widetilde{C_{rh}}$. 
Nonetheless, $\widetilde{C_h}$ is a $\rho^r$-constant field and 
 all results of Section 
\ref{sec: galois} can be applied to  the $(\rho^r,\delta)$-field $\tK$.
\end{rem}

 \section{Auxiliary results}\label{sec: aux}

In this section, we let $(K,\rho)$ be a difference field of characteristic zero and 
 we consider a $\rho$-linear difference system 
 \begin{equation}\label{eq: systg}
 \rho(Y)=AY, \mbox{ with } A \in \GL_n(K) \,,
 \end{equation}
where $\mathbf k:=K^\rho$ is an algebraically closed field.

 \subsection{Irreducible and reducible difference Galois groups}
 
The strategy for proving Theorem \ref{thm: main} is very different according to whether the difference 
Galois group  associated with the difference equation over $K$ satisfied by 
$f$ is irreducible or not. By irreducible, we mean that the group acts irreducibly on $\mathbf{k}^n$, see Definition \ref{def:irreducible}. We point out that this is different from saying that an algebraic group, 
viewed as an algebraic variety, is irreducible.  

\begin{defi}\label{def:irreducible} 
Let $G\subset \GL_{n} (\mathbf{k})$ be a group.  
We say that $G$ is \emph{irreducible} if it acts irreducibly on 
$\mathbf{k}^n$, that is, if the only $\mathbf{k}$-vector subspaces of $\mathbf{k}^{n}$ invariant by $G$ are $\{0\}$ and $\mathbf{k}^{n}$. 
We say that $G$ is \emph{reducible} otherwise.  
When $G$ is an irreducible group, we say that $G$ is \emph{imprimitive} if there exist an integer 
$r\geq 2$, and $V_{1},\dots, V_r$, some $\mathbf{k}$-vector spaces satisfying the following conditions. 
\begin{itemize}
\item[(i)] $\mathbf{k}^{n}=V_1 \oplus \dots \oplus V_r$.  
\item[(ii)] For every  $g\in G$, the mapping $V_i \mapsto g(V_i)$ is a permutation of the set $\{V_{1},\dots,V_r \}$. 
\end{itemize}
\noindent We say that $G$ is \emph{primitive} otherwise. 
\end{defi}
 
The next result shows that a difference Galois group that is irreducible and  imprimitive 
cannot be connected. We recall that an algebraic group $G$ is connected if it has no proper open subgroup with 
respect to the Zariski topology.  The \emph{identity component} of $G$  is the connected component of 
$G$ containing the identity.

\begin{lem}\label{lem6}
 Let  $G \subset \Gl_n(\mathbf{k})$ be an algebraic group. 
 If $G$ is connected and irreducible, then $G$ is primitive. 
\end{lem}

\begin{proof}
Let us assume by contradiction that $G$ is imprimitive. 
Then there exist an integer $r\geq 2$ and some  nonzero $\C$-vector-spaces 
$V_1, \dots, V_r\subsetneq \mathbf{k}^n$ 
such that $\mathbf{k}^n=V_1 \oplus \dots \oplus V_r$. Furthermore, the action induced by every element of $G$ on 
$\mathbf{k}^n$ is 
a permutation of the sets $\{V_1, \dots,V_r\}$. 
Let $G_{V_{1}}$ denote the stabilizer of $V_1$ in $G$. This is an algebraic subgroup of $G$. We claim that it has 
finite index in $G$.  Let us assume by contradiction that there are an infinite number of  cosets 
$g_i G_{V_{1}}$ of $G_{V_{1}}$ in $G$. 
Since the group of permutations of a set with $r$ elements is finite, there exist two distinct cosets 
$g_1 G_{V_{1}}$ and $g_2 G_{V_{1}}$ such that $g_1$ and $g_2$ induce the same permutation of the $V_i$'s.  
In that case, $g_2^{-1}g_1$ stabilizes $V_1$ and thus belongs to $G_{V_{1}}$, providing a contradiction. 
This proves the claim. 
Now, since $G_{V_{1}}$ is a closed subgroup of  $G$ of finite index, it contains the identity component of $G$. The latter is equal to $G$ for $G$ is assumed to be connected.   
 This means that $V_1$ is invariant under the action of $G$. Hence, $G$ is reducible. This provides a 
 contradiction.  
\end{proof}

 Now, let us recall the definition of a difference module, see \cite[\S 1.4]{vdPS97}.  
 
\begin{defi} A $K$-difference module $\cM$  is
a  pair $(M,\Sigma)$ where $M$ is a finite-dimensional $K$-vector space and $\Sigma$ is an 
additive endomorphism of $M$ such that $\Sigma(\lambda Y)= \rho( \lambda) \Sigma(Y)$ for 
every $Y \in M$ 
and $\lambda \in K$.
\end{defi} 
 
Let $A \in \GL_n(K)$. The $K$-difference module $\cM_A$ attached to the system $\rho(Y)=AY$ is 
$\cM_A = (K^n, \Sigma_A)$, where $\Sigma_A(Y) =A^{-1} \rho(Y)$. 
Given a  monic linear difference operator $\cL =a_0 y +a_1 \rho(y) +\dots +\rho^n(y)$ with coefficients $a_i$ in $K$  
and  $a_0 \neq 0$, one can consider the corresponding linear difference system $\rho(Y)=A_\cL Y$, where $A_\cL$ is 
the companion matrix of $\cL$. If $K$ contains  a nonperiodic element with respect to $\rho$, then  \cite[Theorem B.2]{HenSin} asserts that any $K$-difference module $\cM$ is isomorphic to  $\cM_{A_{\cL}}$ for some 
monic linear difference operator $\cL$.  This is equivalent to the existence of a so-called cyclic vector in $\cM$, 
that is, a nonzero element 
$e \in M$ such that  the vectors $e,\Sigma(e),\dots,\Sigma^{n-1}(e)$ form a $K$-basis of $M$. 

\medskip

The next lemma characterizes linear difference systems whose difference Galois group is reducible.

\begin{lem}\label{lem: reducible1}
 Let $G\subset \mathrm{GL}_{n}(\mathbf k)$ be the difference Galois group of  \eqref{eq: systg}
 over $K$ and let $r$ be an integer with $0<r<n$. 
 The following statements are equivalent.
\begin{itemize}
\item There exists a $G$-submodule of $\mathbf k^n$ \footnote{$\mathbf k^n$ is a $G$-module via $G \subset \Gl_n(\mathbf k)$.} of dimension $r$ over $\mathbf k$.
\item There exists a  difference submodule of dimension $r$ in the $K$-difference module $\cM_A$ associated with 
 \eqref{eq: systg}. 
\item There exists $T \in \Gl_n(K)$ such that
 $$\rho(T)AT^{-1} =\begin{pmatrix}
B_1 &B_2 \\
0 & B_3
\end{pmatrix}$$ with $B_1 \in \Gl_r(K)$.
\end{itemize}
In particular, $G$ is reducible if and only if the statements above hold for some $r$, with $0<r<n$.  
Furthermore, if the matrix $A$ in \eqref{eq: systg} is the companion matrix of a difference operator $\cL$, 
then $\cL$ admits a nontrivial right factor if and only if $G$ is reducible.  
\end{lem}

\begin{proof}
Since the category of $K$-difference modules is a  Tannakian category, the first two statements are 
nothing else than the usual Tannakian equivalence, see \cite[Th\'{e}or\`{e}me 3.2.1.1] {andreens}. 
Completing a $K$-basis of a   difference submodule  $\cM$ in $\cM_A$ 
of dimension $r$, $0<r <n$,  into a $K$-basis $\underline{e}=(e_1,\ldots,e_n)$ of $K^n$, shows that 
the matrix of $\Sigma$ in the basis 
$\underline{e}$ is of the form $B=\begin{pmatrix}
B_1 &B_2 \\ 
0 & B_3  
\end{pmatrix}$,  with $B_1 \in \Gl_r(K)$. Let $T \in \Gl_n(K)$ denote the matrix associated with the change of basis from the canonical basis $\underline{f}=(f_1,\ldots,f_n)$ to $\underline{e}$. Then $$BT\underline{f}=B\underline{e}=\Sigma(\underline{e})=\Sigma(T\underline{f})=\rho(T)A\underline{f}.$$ This proves the 
equivalence between the second and the third statements.

Now, let us assume that $A$ is the companion matrix of some difference operator $\cL$ and that there exists a 
nontrivial  difference submodule $\cN$ of $\cM_A$.  Let $e$ be a cyclic vector of $\cM_A$ such that $\cL $ is 
the minimal monic linear difference operator annihilating $e$. 
Since $\cN$ is a nontrivial difference submodule of $\cM_A$, then $\dim_K \cM_A/\cN <n$ and there 
exists a nontrivial monic linear difference operator $\cL_1$ of order smaller than or equal to $\dim_K \cM_A/\cN$ such that $\cL_1 e \in \cN$. The element $\cL_1 e$ is nonzero by minimality of $\cL$. Since $\cN$ is a difference submodule of $\cM_A$ there exists a monic linear difference operator  $\cL_2$ of order smaller than or equal to  $ \dim _K \cN$ such that 
$\cL_2\cL_1(e)=0$. By minimality of $\cL$, this proves that $\cL=\cL_2 \cL_1$, where $ \cL_1$ has order smaller than $n$. 
Conversely, if there is a nontrivial factorization $\cL=\cL_2 \cL_1$, then the difference module corresponding to $\cL_1$ is a nontrivial difference submodule of $\cM$. Hence $G$ is reducible. 
\end{proof}

\subsection{Going up and down}\label{sec:goingupdown}

In this section, we let $K, F_{0}, \rho$ be defined as in Section \ref{sec:mainresult}, and 
we let $\tK$ be defined as in Section \ref{sec: paramframework}.  
The two following lemmas compare the difference Galois groups over $K$ and over $\tK$. 

We recall that a connected linear algebraic group $G$ is said to be \emph{semisimple} if it is of 
positive dimension and has only trivial abelian normal 
subgroups. 
 
\begin{lem}\label{lem:conservationdifferenceGaloisgroup}
Let $A \in\GL_n(K)$ and let $G$ (resp. $\tG$) be the Galois group of $\rho(Y)=AY$ over $K$ (resp. $\tK$). 
The following properties hold.
 
 \begin{itemize}
 \item In Cases  $\bS_0$, $\bS_\infty$, and $\bQ$, we have $\tG=G(\tC)$. 
 \item In Case $\bM$, if $G$ is semisimple, then $\tG=G(\tC)$.
\end{itemize} 
\end{lem}

\begin{proof}
In Cases $\textbf{S}_{0}$, $\textbf{S}_{\infty}$, and \textbf{Q}, the field extension $\tK$  
of $K$ is obtained by extension of the field of constants $\C$ to $\tC$. 
Then, \cite[Corollary~2.5]{chatzidakis2007definitions} gives that the difference 
Galois group over $\tK$ is just the extension of $G$ to the $\tC$-points.

Let us deal with Case \textbf{M}. Let $\mathcal{G}$ denote the difference Galois group over $K(\log)$. As previously, 
we infer from \cite[Corollary~2.5]{chatzidakis2007definitions} that 
$\tG=\mathcal{G}(\widetilde{C})$. Thus, it remains to prove that $\mathcal{G}=G$ when $G$ is semisimple.  
By \cite[Proposition 1.6]{DHR18}, we can see $\mathcal{G}$ as a subgroup of $G$ 
such that  one of the following two situations occurs.  
\begin{itemize}
\item[(1)] $\mathcal{G}=G$. 
\item[(2)] $\mathcal{G}$ is a normal subgroup of $G$ and $G/\mathcal{G}\simeq  \mathbf{G}_m$ the multiplicative group $(\C^*, .)$. 
\end{itemize}
By \cite[Corollary 21.50]{milne2017algebraic}, a semisimple algebraic group is perfect, that is equal to its derived subgroup, see \cite[Definition 12.45]{milne2017algebraic}. Thus, a semisimple algebraic group has no nontrivial abelian  quotient. Hence, (2) 
cannot occur. We deduce that $G=\mathcal{G}$, which ends the proof. 
\end{proof}

\begin{lem}\label{lem: irreducible} Let $A \in \GL_n(K)$ and let $G$ (resp.\ $\widetilde{G}$) be the Galois group of $\rho(Y)=AY$ over $K$ (resp.\ $\tK$). 
The group $G$ is irreducible if and only if the group $\tG$ is irreducible.
\end{lem}

\begin{proof}
 In Cases $\bS_0$, $\textbf{S}_{\infty}$, and \textbf{Q}, Lemma \ref{lem:conservationdifferenceGaloisgroup} 
 ensures that $\tG=G(\widetilde{C})$ and the result follows directly.

Let us deal with Case \textbf{M}. 
We infer from \cite[Corollary~2.5]{chatzidakis2007definitions}  that 
$\tG=\mathcal{G}(\widetilde{C})$, where we let $\mathcal{G}$ denote the Galois group of the system over $K(\log)$. 
Thus, it remains to prove that $\mathcal{G}$ is irreducible if and only if 
$G$ is irreducible.  
As seen in the proof of Lemma~\ref{lem:conservationdifferenceGaloisgroup},  $\mathcal{G}$ is a subgroup of $G$. 
Thus,  
 $\mathcal{G}$ is reducible as soon as $G$ is reducible. Conversely, let us assume that $\mathcal{G}$ is reducible. 
By the cyclic vector lemma over $K$, the linear difference system \eqref{eq: systg} is equivalent to an equation 
$$
\mathcal{L}:=\rho^{n}+a_{n-1}\rho^{n-1}+\dots+a_0=0 \,,
$$ with $a_i\in K$ and $a_0\neq 0$. 
Since $\mathcal{G}$ is reducible, Lemma \ref{lem: reducible1} gives that $\mathcal{L}$ admits a factorization 
over $K(\log)$. Multiplying $\mathcal{L} $ by some nonzero element  if necessary, we can get rid of 
the denominators. Then, there exists $\alpha \in K[\log]$ such that  
$$
\alpha \mathcal{L}  =(b_k \rho^{k}+b_{k-1}\rho^{k-1}+\dots+b_0)(c_{n-k} \rho^{n-k}+c_{n-k-1}\rho^{n-k-1}+\dots+c_0) \,,
$$
where $b_i, c_i\in K[\log]$, $b_k b_0 c_{n-k} c_0\neq 0$, and $0<k<n$.  
 Rewriting the equation in terms of powers of log, one finds 
\begin{eqnarray*}
\sum_{ i\geq i_0} \log^i (\alpha_i \cL) & =& \left(\sum_{\kappa \geq \kappa_0} \log^\kappa \cL_\kappa \right) \left(\sum_{j\geq j_0} \log^j \cD_j \right) \,,
\end{eqnarray*}
where $\cL_\kappa, \cD_j$ are linear difference operators over $K$, $\alpha_{i} \in K$,  $\alpha_{i_0} \in K^*$,
and $\cL_{\kappa_0}:=\sum_{r=0}^t \gamma_r \rho^r$ and $\cD_{j_0}$ are nonzero. Using the fact that $\log$ is transcendental over $K$ and $\rho \log =p \log \rho$, we deduce that $i_0= \kappa_0+j_0$. 
Considering the terms in $\log^{i_{0}}$ in both sides of the factorization, we deduce that $\cL =\frac{1}{\alpha_{i_0}} \cM\cD_{j_0}$  where $\cM=\sum_{r=0}^t \gamma_rp^{j_0 r} \rho^r$, leading to a nontrivial factorization of $\cL$ over 
$K$. By Lemma \ref{lem: reducible1}, we obtain that $G$ is reducible, as wanted. 
\end{proof}

Now, we prove the following descent lemma.

\begin{lem}\label{lem:intersectionF0tK}
For $F_0,K$ and $\tK$ defined as in Section \ref{sec: paramframework}, we have 
$$F_0 \cap \tK =K.$$
\end{lem}

\begin{proof}
Let us first consider Case $ \textbf{Q}$. Let $a \in F_0 \cap \tC(x^{1/*})$. 
There exists a positive integer $r$ such that $a \in \C((x^{1/r})) \cap \tC(x^{1/r})$.  
It follows that there exists  a positive integer $N$ such that $ x^N a \in \C[[ x^{1/r}]]$. 
Thus,  $x^Na$ is a formal power series in $x^{1/r}$ with complex coefficients that represents  a rational fraction  with coefficients in $\tC$.  This property can be characterized by the vanishing of its associated Kronecker-Hankel  determinants (see, for instance, \cite[p.~5]{Salem}). This condition does not depend on the field of coefficients, so 
we deduce that $x^Na \in \C(x^{1/r})$. Hence, $a\in K$.   

The proof of Case $\bS_\infty$ is entirely similar.  For Case $\bM$,  
if $a \in \C((x^{1/*})) \cap \tC(x^{1/*})(\log)$, we use the fact that the logarithm is transcendental over $\C ((x^{1/*}))$ to conclude that  $a \in \C((x^{1/*})) \cap \tC(x^{1/*})$. Arguing as above, we can show that $a \in K$.  

Now, let us deal with Case $\textbf{S}_0$. Let $a \in F_0 \cap \tC(x)$. Since $F_0 \subset \cM er(\C)$, 
one can find $x_0 \in \C$ such that $a \in \C[[x-x_0]]$. 
Once again, using Kronecker-Hankel determinants, one can deduce that $a \in K$.
\end{proof}

\subsection{Iterating the difference operator}

For every positive integer 
 $\ell$, let us set $${A_{[\ell]}:=\rho^{\ell-1}(A)\times\dots\times A\in \mathrm{GL}_{n}(K)}\,.$$ 
Note that if $U$ is a fundamental matrix of  $\rho (Y)=AY $, then it is also a fundamental matrix of   
$\rho^{\ell}(Y)=A_{[\ell]} Y$. In other words, a vector solution to the system $\rho (Y)=AY$ is also a solution to the 
iterated system ${\rho^{\ell}(Y)=A_{[\ell]} Y}$. 
In this section, we show that, replacing the original system by a suitable iteration, we can  
reduce the situation to the more convenient case where the corresponding Galois group is connected and the 
Picard-Vessiot extension is a field. 

This iteration process might introduce
some new constants. The following lemma shows that  this cannot happen when the fields of $\rho$-constants is algebraically closed. 

\begin{lem}\label{lem:constantiterate}
 Let $L$ be a $\rho$-field such that $L^\rho=\mathbf{k}$, and $r$ be a positive integer. 
 Then, any $\rho^r$-constant of $L$ 
 is algebraic of degree less than or equal to $r$ over $\mathbf{k}$. 
 In particular, if $\mathbf{k}$ is algebraically closed then $L^{\rho^r}=\mathbf{k}$.
\end{lem}

\begin{proof} 
Let $a \in L$ be a $\rho^r$-constant. Then the polynomial $P(X)=(X-a)\hdots (X- \rho^{r-1}(a))$ 
belongs to $L^{\rho}[X]=\mathbf{k}[X]$. This proves that  $a$ is algebraic over $\mathbf{k}$  of degree 
less than or equal to $r$. 
\end{proof}

One of the main  difficulty of difference algebra is the control of the algebraic  difference field extensions. This is   the source of many technical difficulties
and connected to the fact that  the theory ACFA is unstable (see \cite{ACFA}).
The difference fields $(K,\rho)$ introduced in Section \ref{sec:mainresult}  have the following very  strong property with respect to finite difference fields extensions, which will allow us to bypass these difficulties.
 
\begin{lem}\label{lem: C1} 
Let $r$ be a positive integer. 
In each of Cases $\textbf{S}_{0}$, $\textbf{S}_{\infty}$, $\textbf{Q}$, and $\textbf{M}$,  
there are no proper $\rho^r$-field extension of $K$ of finite degree. 
\end{lem}

\begin{proof}
The results in \cite{hendriks1997algorithm, 
 hendriks1998algorithm,Ro18} cover all 
 cases.  
\end{proof}

We are now ready to prove the following result. 

 \begin{propo}\label{prop: diffGaloisconnexe}
Let $(K,\rho)$ be defined as in Section \ref{sec:mainresult} and let $L$ be a $\rho$-field extension of $K$ 
such that $L^\rho=\C$.  Let  $(f_1, \dots , f_n)^{\top}\in L^{n}$ be a nonzero solution to \eqref{eq: systg}. 
Let $G$ be the difference Galois group of \eqref{eq: systg} over $K$. There exists a positive integer 
$r$ such that  the following properties hold. 
 
\begin{trivlist}
\item[(a)] There exists a
 Picard-Vessiot field  extension $\mathcal Q$ for $\rho^{r}(Y)=A_{[r]}Y$ over $(K,\rho^r)$, with a fundamental  solution  matrix $U$ 
 having $(f_1, \dots , f_n)^{\top}$ as first column.
 
 \item[(b)]  The difference Galois group $G_r:=\Gal(\mathcal Q/K)$   of $\rho^{r}(Y)=A_{[r]}Y$ over $(K,\rho^r)$ 
  coincides with the identity component of $G$ and is therefore connected. 
\end{trivlist}
 \end{propo}

Though this result can be easily deduced from the proof of \cite[Proposition~4.6]{DHR2} 
in the particular case where the parametric operator is the identity, we find more convenient for the reader 
to include the proof below.  

\begin{proof}

(a) Since $K ( f_1,\dots,f_n )^{\rho}\subset L^{\rho}=\C$,  
 there exists a Picard-Vessiot ring $R_0$ for $\rho(Y)=AY$ over $K ( f_1,\dots,f_n )$. 
Since $f_1,\dots,f_n\in  R_0$, there exists a fundamental matrix $U$  for $R_0$ having 
$(f_1, \dots , f_n)^{\top}$ as first column.  Let $\cQ_0$ be the total quotient ring of $R_0$. By Proposition \ref{prop:PvringPvext}, we find
$\cQ_0^{\rho}=\C$. 

Set $R:=K[ U,1/\det(U)]$. The difference ring $(R,\rho)$ is a subring of the Picard-Vessiot ring $(R_0,\rho)$.
  Since any zero divisor in $R$ is a zero divisor in $R_0$, the total quotient ring $\cQ$ of $R$ embeds in $\cQ_0$  and consequently is a pseudofield (see \cite[Lemma 1.3.4]{wibmerthesis}). Consequently $\C \subset \cQ^{\rho} \subset \cQ_0^{\rho}=\C$.  
  Then $\cQ^{\rho}=\C$, and $(\cQ,\rho)$ is a Picard-Vessiot extension for $\rho(Y)=A Y$ over $(K,\rho)$.  Since $K^\rho=\C$ is algebraically closed, the uniqueness of Picard-Vessiot extensions above $K$ allows to conclude that $\Gal(\cQ/K)=G$.  We have proved that we can embed any solution in $L$ in a Picard-Vessiot extension above $K$.

By Proposition \ref{prop:PvringPvext}, $R$ is a Picard-Vessiot ring for $\rho(Y)=A Y$ over $(K,\rho)$.
Let $e_0,\dots,e_{r-1}$ denote the 
orthogonal idempotent relative to its ring structure.  
Since $K[f_1,\dots,f_n]$ is an integral domain contained in $R$, which is  the direct sum of the integral domains 
$e_iR$, there exists $i \in \{0,\dots,r-1\}$ such that $K[f_1,\dots,f_n] \subset e_i R$.  
Since $K^\rho=\C$ is algebraically closed,  Lemma \ref{lem:constantiterate} implies that $K^{\rho^r}=\C$.
By \cite[Lemma 1.26]{vdPS97}, $(e_{i}R,\rho^{r})$ is a Picard-Vessiot ring for the $\rho^{r}$-system  
$\rho^{r}(Y)=A_{[r]} Y$ and $e_i R$ is an integral domain. Its  total quotient ring 
 $\mathcal{Q}_1$ is a field and, since $K^{\rho^r}$ is algebraically closed, $\cQ_1$ is a 
 Picard-Vessiot field extension by  Proposition \ref{prop:PvringPvext}.  By construction, $\cQ_1$  
 contains $f_1,\dots,f_n$, 
and we can thus choose a fundamental matrix having $(f_1, \dots , f_n)^{\top}$ as first column.

\bigskip

\noindent 
(b) By \cite[Theorem 12]{Ro18},  $G_{r}=\Gal(\cQ_1/K)$ is a normal algebraic subgroup of $G$ and the quotient $G/G_{r}$ is finite. To conclude, it remains to prove that $G_{r}$ is connected. Let $G_{r}^{\circ}$ denote its identity component. The 
Galois correspondence, see Theorem \ref{thm: correspondence}, gives that the 
$\rho^r$-field ${\mathcal{Q}_1}^{G_{r}^{\circ}}$ is a finite extension of $K$. 
Lemma \ref{lem: C1} implies that ${\mathcal{Q}_1}^{G^{\circ}_{r}} =K$ and, 
applying the Galois correspondence again, we deduce   
that ${G^{\circ}_{r}}=G_{r}$. 
\end{proof}

\begin{rem}\label{rem:groupofiteratesifconnected}
The proof of (b) shows that, if $G$ is connected, then, for every positive integer $\ell$, the Galois group of 
$\rho^\ell(Y)=A_{[\ell]} Y$ over $K$ coincides with $G$. 
\end{rem}

In the parametrized framework, we have the following similar result.  

\begin{propo}\label{prop:goodPPVsolutionabstractdeltafield}
Let  $L$ be a $(\rho,\delta)$-field extension of $\tK$ with $L^\rho=\tC$, and  
let $(f_1, \dots , f_n)^{\top}\in L^{n}$ be a nonzero solution to \eqref{eq: systg}. 
 Then there exist a positive integer $r$  and a 
  parametrized  Picard-Vessiot field extension $\widetilde{\mathcal Q}$ for $\rho^{r}(Y)=A_{[r]}Y$ over $\tK$, with a fundamental matrix $U$ 
 having $(f_1, \dots , f_n)^{\top}$ as first column.
 \end{propo}

\begin{proof}
The proof  is similar to  Proposition \ref{prop: diffGaloisconnexe}   and  \cite[Lemma 3.7]{DHR18}, just noting that, 
by Lemma \ref{lem:constantiterate}, $\tK^{\rho^r}=\tC$ for $\tC$ is algebraically closed. 
\end{proof}

The following result is obtained by combining the two previous propositions.

\begin{coro}\label{cor:iterationGaloisgroupconnectedPPVfield}
Let  $(f_1, \dots , f_n)^{\top}\in F_{0}^{n}$ be a nonzero solution to \eqref{eq: systg} and let $G$ 
be the difference Galois group of \eqref{eq: systg}. Then, there exists a positive integer $r$ such that 
\begin{itemize}
\item[(a)] There exists a
$(\rho^r, \delta)$-Picard-Vessiot \emph{field} extension $\widetilde{\mathcal Q}$ for ${\rho^{r}(Y)=A_{[r]}Y}$  with a fundamental matrix $U$ 
 having $(f_1, \dots , f_n)^{\top}$ as first column over $\tK$ (where $\tK =\widetilde{C_{sh}}(x)$ for some positive integer  $s$ dividing $r$ in Case $\bS_0$). 
 \item[(b)] The difference Galois group of $\rho^{r}(Y)=A_{[r]}Y$ over $K$ coincides with the identity  component of $G$.
\end{itemize}

\end{coro}
\begin{proof}

By Proposition \ref{prop: diffGaloisconnexe}, there exists a positive integer $s$ such that 
the difference Galois group of $\rho^s(Y)=A_{[s]}Y$ is equal to the identity component of $G$.  
In order to apply Proposition \ref{prop:goodPPVsolutionabstractdeltafield} to the system $\rho^s(Y)=A_{[s]}Y$, 
we need to embed $F_0$ into a $(\rho^s,\delta)$-field extension $L$ of $\tK$ with $L^{\rho^s}=\tC$. 
We proceed as follows. 

\begin{itemize}

\item Case ${\bS}_0$. We let $F_1$ denote the smallest $(\rho^s,\delta)$-subfield of $\cM er(\C)$ containing $F_0$ and $C_{sh}$. Since $\cM er(\C)^{\rho^s}=C_{sh}$, we find that
 ${F_1}^{\rho^s}=C_{sh}$. Let $\tC$ denote a $\delta$-closure of $C_{sh}$, considered as a constant $\rho^s$-field. 
 The fields extensions $\tC$ and $F_1$ are linearly disjoint above $C_{sh}$ (see \cite[Lemma 1.1.6]{wibmerthesis}).  
 Thereby, their compositum $L$ is the fraction field of $\tC \otimes_{C_{sh}} F_1$ (see \cite[A.V.13]{Bourbakialgebra}).  
 Arguing as in  the proof of \cite[Lemma 2.3]{DHR18}, we obtain that $L$  has the required properties.

 \item Cases $\bS_\infty$ and $\bQ$. We let $L$ be the fraction field of $F_0 \otimes_\C \tC$.

 \item Case $\textbf M$. We let $L$ be the field $\tC((x^{1/*}))(\log)$. 
\end{itemize}

Applying Proposition \ref{prop:goodPPVsolutionabstractdeltafield} to the system $\rho^s(Y)=A_{[s]}Y$ over $\tK$, one can perform a second iteration  to obtain a system $\rho^{sl}(Y)=A_{[sl]}Y$ satisfying $(a)$. By Remark~\ref{rem:groupofiteratesifconnected}, the difference Galois group of $\rho^{sl}(Y)=A_{[sl]}Y$ over $K$ is the identity component of $G$. This ends the proof. 
\end{proof}
 \section{Proof of Theorem \ref{thm: main}}\label{sec: main}

This section is devoted to the proof of our main result. 
Before proving Theorem~\ref{thm: main} in full generality, we first consider the following 
two particular cases. 
\begin{itemize}
\item[-] The function $f$ is solution to an inhomogeneous equation of order one. 
\item[-] The difference Galois group of the equation associated with 
$f$ is both connected and irreducible. 
\end{itemize}

All along this section, we keep on with the notation of Sections \ref{sec:mainresult} and \ref{sec: paramframework}. 
\subsection{Affine order one equations}

Various hypertranscendence criteria for solutions to inhomogeneous order one 
equations have already been obtained, see for instance \cite{rande1992equations,ishizaki1998hypertranscendency,HS08,NG12}. 
We first deduce from these criteria the following result.

\begin{propo} \label{prop: order1}
If $f\in F_{0}$ is solution to an equation of the form ${\rho(f) = a f + b}$,  
with $ a, b \in K$, then either $f$ is $\tr$-transcendental over $K$ or $f$ belongs to $K$.
\end{propo} 

\begin{proof}[Proof of Proposition \ref{prop: order1}] 
Let us first note that if $a=0$, then $f=\rho^{-1}(b)\in K$.  
We can thus assume that $a\neq 0$. 
Furthermore, we observe that, in Cases $\textbf{Q}$ and $\textbf{M}$, we can use a change of variable 
of the form $z=x^{1/\ell}$ to  reduce the situation to the case where $K=\C(x)$.  
From now on, we thus assume that $K=\C(x)$.  
Let us also assume that $f$ is $\tr$-algebraic over $K$.  
It remains to prove that $f\in K$.

We first note that Case $\textbf{M}$ corresponds to 
\cite[Th\'eor\`eme 5.2]{rande1992equations}. 
 Now, let us consider Cases $\textbf{S}_{0}$, $\textbf{S}_{\infty}$, and $\textbf{Q}$, which  are partially covered by  
\cite{HS08}.  
We recall the following definition from \cite[\S 6.1]{HS08} 
(see also \cite[\S 2]{vdPS97}). 
 We say that $\alpha \in K$ is \emph{standard} if, for all positive integers $\ell$, 
$\alpha$ and $\rho^\ell(\alpha)$ have no common zero or pole.  
By \cite[Lemma~6.2]{HS08}, there exists a standard element $a^*$ and a nonzero $e \in K$ such that 
$a =a^* \rho(e)/e$.  
It follows that $f/e$ is solution to the equation 
$$
\rho(f/e)=a^*f/e+b/\rho(e) \,.
$$ 
Set $g=f/e$ and $\tilde{b}=b/\rho(e)$. Since $e\in K$, it suffices to prove that 
$g\in K$. 
Let us note that $g=f/e$ is  $\tr$-algebraic over $K$. 
 
  \medskip
  
In Cases $\bS_0$ and $\bS_\infty$,  \cite[Proposition~3.9]{HS08} implies that $a^{*} \in \C$ 
and there exists $h \in K$ such that 
 $\tilde{b} = \rho(h) - a^{*} h$. We obtain that 
\begin{eqnarray*}
\rho (g)&=&a^{*}g+\tilde{b}\\
&=&a^{*}g+\rho(h) - a^{*}h \,.
\end{eqnarray*}
Hence $\rho(g-h)=a^{*}(g-h)$. If $g=h$, then $g\in K$ and we are done. Thus, we can assume that $g\neq h$. 
In that case, setting $\tilde{f}=g-h$, we obtain $\rho(\tr(\tilde{f}))=a^{*}\tr(\tilde{f})$. It follows that 
$\rho (\tr(\tilde{f})/\tilde{f})=\tr(\tilde{f})/\tilde{f}$, and thus $\tr(\tilde{f})/\tilde{f}\in \C\subset K$. 
Since $\tilde{f}\in F_{0}$ satisfies both a linear $\rho$-equation and a linear $\tr$-equation with coefficients 
in $K$,  Theorem \ref{lem: linlin} implies that $\tilde{f}\in K$. Since $h\in K$, we deduce that $g\in K$, as expected. 
  
  \medskip

In Case $\textbf{Q}$, \cite[Proposition~3.10]{HS08} implies  that  
$a^{*} = cx^\alpha$ for some $c \in \C^{*}$ and $\alpha \in \Z$, and one of the following conditions holds. 
   
\begin{enumerate}
\item[(a)] $a^{*} = q^r$ for some $r \in \Z$, and $\tilde{b} = \rho(h) - a^{*}h + dx^r$ for some $h \in K$ and $d \in \C$.  
\item[(b)] $a^{*} \notin q^{\Z}$ and $\tilde{b} = \rho(h) - a^{*}h $ for some $h \in K$.
\end{enumerate}

In Case (a),  we deduce that 
\begin{equation}\label{eq: caseQa}
\rho\left(\frac{g-h}{x^{r}}\right)=\frac{g-h}{x^{r}}+dq^{-r}\,\cdot
\end{equation}
Since $(g-h)/x^{r}\in F_{0}$, we can consider its expansion in Puiseux series. 
Since $q$ is not a root of unity, we easily deduce from Equation \eqref{eq: caseQa} that $(g-h)/x^r\in \C\subset K$. 
Hence $g\in K$, as wanted.

In Case (b), we have that $a^{*} \notin q^{\Z}$ and  $\tilde{b} = \rho(h) - a^{*}h$ for some $h \in K$. 
It follows that  $\rho (g-h)=cx^{\alpha}(g-h)$, where 
 $g-h\in F_0$ is a Puiseux series. Arguing as  in Case (a), we obtain that $\alpha=0$ 
and $g-h=c_{0}x^{r}$ for some $r\in \Q$ and $c_{0}\in \C$. Hence $g\in K$.   This ends the proof. 
\end{proof}

For order one equations over $K$ whose solutions do not necessarily belong to  $F_{0}$, 
one has the following criterion for differential algebraicity.   

\begin{lem}\label{lem:rankonehomog}
Let $L$ be a $(\rho,\delta)$-field extension of $\tK$ with $L^\rho=\tC$. Let $f \in L^*$
such that $\rho(f)=a f$ with $a \in K^*$. Then, $f$ is $\delta$-algebraic over $\tK$ if and only if  
$a =cx^\alpha \frac{\rho(g)}{g}$ for some $c\in\C^*$, $\alpha \in \Q$, and $g \in K^*$. 
Moreover, $\alpha=0$ in Cases $\bS_0$, $\bS_\infty$, and $\bM$.
\end{lem}

\begin{proof}
Cases $\bS_0$, $\bS_\infty$, and $\bQ$ are given by \cite[Corollary 3.4]{HS08}, while Case 
$\bM$ is given by \cite[Proposition~3.1]{DHR18} combined with  some descent argument  to go from 
$\tC$ to $\C$. Such argument is similar to the one given in  
the proof of 
\cite[Corollary 3.2]{HS08}.
\end{proof}

\subsection{Connected and irreducible Galois groups}\label{sec:connexe}

As a second step, we consider the situation where the difference Galois group 
is both connected and irreducible.

\begin{propo} \label{prop: irreducible}  
Let us assume that $n\geq 2$. If $f\in F_0$ is a nonzero solution to Equation \eqref{eq: difference} and if  
the corresponding difference Galois group $G$ over $K$ is both 
connected and irreducible, then $f$ is $\tr$-transcendental over $K$. 
\end{propo}

Proposition \ref{prop: irreducible} is obtained as a direct consequence of the following more general result.  
 
\begin{propo}\label{prop: colonne}
Let us consider a linear system of the form \eqref{eq: systg} with $n\geq 2$. 
Let us assume that the difference Galois 
group $G$ for \eqref{eq: systg} over $K$ is connected and irreducible.  
Let $\widetilde{\cQ}$ be a $(\rho,\delta)$-Picard-Vessiot field extension for \eqref{eq: systg} 
over $\widetilde K$, with fundamental matrix of solutions $U$. Then every column of $U$ 
contains at least one element that is 
$\delta$-transcendental over $\widetilde{K}$. 
\end{propo}

\begin{proof}[Proof of Proposition \ref{prop: irreducible}] 
We argue by contradiction, assuming that $f$ is $\tr$-algebraic over $K$.  
Then all coordinates of the vector $(f,\dots, \rho^{n-1} (f))^{\top}$ are also $\tr$-algebraic over $K$. 
This follows from the fact that $\rho$ and $\tr$ almost commute, that is 
$\tr \rho  = c \rho \tr $ with $c=1$ in Cases $\bS_0$, $\bS_\infty$, and $\bQ$, 
and $c=p$ in Case $\bM$. We deduce from the construction of $\delta$ with respect to $\tr$ given in Section 
\ref{sec: paramframework} (and the fact that $\log$ is $\delta$-algebraic for Case $\bM$)  
that all coordinates of the vector $(f,\dots, \rho^{n-1} (f))^{\top}$ are also 
$\delta$-algebraic over $\tK$. 

Now, let $A$ denote the companion matrix associated with Equation \eqref{eq: difference}. 
Since $(f,\dots,\rho^{n-1}(f))^{\top}\in F_0^n$ is nonzero, 
Corollary \ref{cor:iterationGaloisgroupconnectedPPVfield} ensures the existence of a positive integer $r$ 
and a  $(\rho^{r},\delta)$-Picard-Vessiot field extension $\widetilde{\cQ}$ for $\rho^{r}(Y)=A_{[r]}Y$ over $\tK$ 
such that  the vector $(f,\dots,\rho^{n-1}(f))^{\top}$ is 
the first column of a fundamental matrix $U$. Furthermore,  Remark \ref{rem:groupofiteratesifconnected}  ensures that 
the Galois group of $\rho^r(Y)=A_{[r]}(Y)$ over $K$ 
is equal to $G$, for the latter is connected.  Thus, 
Proposition~\ref{prop: colonne} applies with $\rho$ replaced by $\rho^r$, providing a contradiction. 
\end{proof}

It remains to prove Proposition \ref{prop: colonne}. As a key argument, we will use the following result due to 
Arreche and Singer \cite[Lemma~5.1]{AS17}.  It says that the parametrized Galois group must be as big as possible when 
the difference Galois group has an identity component that is semisimple. 
Let us recall that a connected algebraic group is said to be  \emph{semisimple} if it is of 
positive dimension and has only trivial abelian normal 
subgroups.

\begin{propo}\label{propo:AS}
Let us consider a linear system $\rho (Y)=AY$, where 
 $A\in \mathrm{GL}_{n}(\C(x))$ in Cases $\mathbf S_0$ and $\mathbf S_{\infty}$, and where 
$A\in \mathrm{GL}_{n}(\C(x^{1/\ell}))$  for some positive integer $\ell$ in Cases $\mathbf Q$ and $\mathbf M$.  
In Cases $\mathbf S_0$, $\mathbf S_{\infty}$, and $\mathbf M$, we let $G\subset \mathrm{GL}_{n}(\widetilde{C})$ 
denote the difference Galois group over $\widetilde{K}$, and $H\subset \mathrm{GL}_{n}(\widetilde{C})$ denote the $(\rho,\delta)$-Galois group over $\widetilde{K}$. 
In Case $\mathbf Q$, we let  $G\subset \mathrm{GL}_{n}(\widetilde{C})$ 
denote the difference Galois group over $\tC(x^{1/\ell})$, and 
$H\subset \mathrm{GL}_{n}(\widetilde{C})$ denote the $(\rho,\delta)$-Galois group over $\tC(x^{1/\ell})$. 
 As in Proposition \ref{prop: paradense}, we  see $H$ as a subgroup of $G$.  If the identity component of $G$ 
 is semisimple, then $H=G$.
\end{propo}

\begin{proof} 
Cases $\mathbf S_0$ and $\mathbf S_{\infty}$ are explicitly proved in \cite[Lemma~5.1]{AS17}, since, by Remark \ref{rem:diffconstant}, the field of  differential constants of $\tC$ is $\C$. Cases 
$\mathbf Q$ and $\mathbf M$ are not explicitly proved in  \cite[Lemma~5.1]{AS17} for we consider the slightly more general situation where  
$A\in \mathrm{GL}_{n}(\C(x^{1/\ell}))$. However, the result follows easily from the argument given  in the proof of 
 \cite[Lemma~5.1]{AS17}. 
\end{proof}

We are now ready to prove Proposition  \ref{prop: colonne}.

\begin{proof}[Proof of Proposition \ref{prop: colonne}]  

We let $H$ denote the $(\rho,\delta)$-Galois group of  \eqref{eq: systg} over $\widetilde{K}$, 
that is 
$$
H=\Gal^{\delta}(\widetilde{\mathcal{Q}}/\widetilde K)\subset \mathrm{GL}_{n}(\widetilde{C})\,.
$$ 
We also let $\tG$ denote the Galois group of $\rho(Y)=AY$ over $\tK$.
Let us argue by contradiction, assuming the existence of one  column of $U$ whose coordinates are all 
$\delta$-algebraic over $\widetilde{K}$. 

We first show that all entries of $U$ are 
$\delta$-algebraic over $\widetilde{K}$. 
Since $G$ is irreducible, Lemma \ref{lem: irreducible} implies that $\tG$ is irreducible too.  
By Proposition~\ref{prop: paradense}, $H$  
is Zariski dense in the irreducible group $\tG$. This implies that $H$ is irreducible, since 
otherwise, $H$ would be conjugated to a group formed by block upper triangular matrices, as well as its 
Zariski closure, contradicting the irreducibility of $\tG$.  
Now, let $\mathcal S$ denote the 
 $\widetilde{C}$-vector space generated by all columns of $U$ whose 
 coordinates are all $\delta$-algebraic over $\widetilde{K}$. By assumption, $\mathcal S$ is not reduced to $\{0\}$. 
 Furthermore, $\mathcal S$ is invariant under the action of any 
 $\sigma\in H$.  
 Since $H$ is irreducible, we deduce that $\mathcal{S}$ contains $n$ 
 linearly independent solutions to \eqref{eq: systg}. In other words, all columns of $U$ belong to $\mathcal S$, which  
implies that all entries of $U$ are $\delta$-algebraic over $\widetilde{K}$, as claimed. 
 
We observe that the determinant $\det(U)$ is solution to the equation 
\begin{equation}\label{eq: detA}
\rho(y)=\det(A)y
\end{equation}
and  that 
the difference Galois group of this equation over 
$K$ is the group $\det(G)$. Since all entries of $U$ are $\delta$-algebraic over $\widetilde{K}$, 
we obtain that $\det(U)$ is also $\delta$-algebraic over $\widetilde{K}$. 
We first consider the particular case where ${\det(U)\in K}$, and then we move to the general case. 

\medskip

Let us first assume that $\det(U)$ belongs to $K$. In that case, 
$\det(G)=\{1\}$ and therefore $G\subset \mathrm{SL}_{n}(\C)$. 
We recall that $G$ is assumed to be connected and irreducible. 
According to Lemma \ref{lem6}, 
we have that $G$ is primitive. 
By \cite[Proposition 2.3]{SU93}, we finally obtain that $G$ is semisimple. 
By Lemma~\ref{lem:conservationdifferenceGaloisgroup}, $\tG=G(\widetilde{C})$ is semisimple too.

In Cases $\textbf{S}_{0}$, $\textbf{S}_{\infty}$, and $\textbf{M}$, we infer from Proposition \ref{propo:AS} 
that $H=\tG$. 
Since all entries of $U$ 
are $\delta$-algebraic over $\widetilde{K}$, the $\delta$-dimension of the $(\rho,\delta)$-Picard-Vessiot extension  
is zero\footnote{We refer the reader to \cite[P. 374]{HS08} for a definition of the notion of $\delta$-dimension of a 
$(\rho,\delta)$-Picard-Vessiot ring and of a $(\rho,\delta)$-Galois group.}, and therefore the $\delta$-dimension of 
$H$ is zero by \cite[Proposition~6.26]{HS08}. 
By a result of Kolchin \cite[Chap. IV, Proposition~10]{Ko73}, the dimension of an algebraic group $\mathcal{G}$ over $\tC$ is the same as the $\delta$-dimension of $\mathcal{G}$ viewed as a differential group over $\tC$. This proves that the dimension of $H$ equals the dimension of $G(\tC)$ . Since $\C$ is algebraically closed, the dimension of $G(\tC)$ over $\tC$ equals the dimension of $G$ over $\C$. 
Thus, the above equality $H=\tG=G(\widetilde{C})$ implies 
that the dimension of the algebraic group $G$ over $\C$ is zero.  
We recall that an algebraic group of dimension zero is just a finite group, and that a finite connected group 
has cardinality one. 
Thus, we deduce that $G=\{ \mathrm{I}_{n}\}$, 
where we let $\mathrm{I}_{n}$ denote the identity 
matrix of size $n$. Since by assumption $n\geq 2$, this provides a contradiction with the assumption 
that $G$ is irreducible. This ends the proof in these cases.

In Case  $\bQ$, the system $\rho (Y)=AY$ has coefficients in $\C(x^{1/\ell})$ and the entries of $U$ are 
$\delta$-algebraic over  $\widetilde{C}(x^{1/\ell})$, for some positive integer $\ell$.  
Let $G_{\widetilde{K},\ell}$ denote the difference Galois group of $\rho (Y)=AY$ over $\widetilde{C}(x^{1/\ell})$. 
 Since $\widetilde{K}$ is an algebraic extension of $\widetilde{C}(x^{1/\ell})$,  we deduce from \cite[Theorem~7]{Ro18} 
  that $G_{\widetilde{K}}$ and $G_{\widetilde{K},\ell}$ have the same identity component, which is $G_{\widetilde{K}}$ since the latter is connected. Hence the identity component of the identity of 
  $G_{\widetilde{K},\ell}$ is semisimple. By Proposition \ref{propo:AS}, it follows that
 $H_{\ell}=G_{\widetilde{K},\ell}$, where  $H_{\ell}$ denote the $(\rho,\delta)$-Galois group over 
$\widetilde{C}(x^{1/\ell})$. 
 Furthermore, the entries of $U$ are $\delta$-algebraic over $\widetilde{C}(x^{1/\ell})$ 
which  implies that the $\delta$-dimension of $H_{\ell}$ is zero. 
Since 
$H_{\ell}=G_{\widetilde{K},\ell}$, we get as previously that the algebraic group $G_{\widetilde{K},\ell}$ has dimension zero. 
Since $G_{\widetilde{K}}\subset G_{\widetilde{K},\ell}$, this proves that the dimension of $G_{\widetilde{K}}$ 
is also zero, and we can argue as previously to get a contradiction.   

\medskip

Now, let us consider the general case. We remind that $\det(U)$ is $\delta$-algebraic over $\widetilde{K}$ and $\rho(\det(U))=\det(A)\det(U)$. By Lemma \ref{lem:rankonehomog}, there exist some rational number $\alpha$ and nonzero elements 
$c\in \C$ and $g\in K$ such that $\det(A)=cx^{\alpha}\rho(g)/g$. 
Furthermore, $\alpha=0$ in Cases $\textbf{S}_{0}$, 
$\textbf{S}_{\infty}$, and \textbf{M}.

 Let us consider the rank one linear difference system 
\begin{equation}\label{eq:det}
\rho (y) = c^{-1/n}x^{-\alpha/n} y
\end{equation}

 Since $\tcQ$ is a $(\rho,\delta)$-field with $\tcQ^{\rho}=\widetilde{C}$, there exists 
a $(\rho,\delta)$-Picard-Vessiot extension  $\tcQ_1$ for  \eqref{eq:det} over $\tcQ$.   
Let  $(\lambda)\in\GL_1(\tcQ_1)$ be a fundamental matrix associated with this system. 
Then  
$$
\rho(\lambda)=c^{-1/n}x^{-\alpha/n}\lambda \,
$$
and $\lambda$ is invertible in $\tcQ_1$. 
Using the commutativity of $\delta$ and $\rho$, we obtain that 
$$
\delta \left(\frac{\delta(\lambda)}{\lambda}\right) \in \tcQ_1^\rho= \widetilde{C} \,, 
$$
which shows that $\lambda$ is $\delta$-algebraic 
 over $\widetilde{C}$. In particular, it is    $\delta$-algebraic 
 over $\widetilde{K}$.
 Thus, all entries of the matrix $\lambda U \in \GL_n(\tcQ_1)$ are $\delta$-algebraic over 
 $\widetilde{K}$. 
 
On the other hand, we set 
$$
B=c^{-1/n}x^{-\alpha/n}A  \in \Gl_n(K) \;\mbox{ and } U_B:=\lambda U.
$$ 
Note that $\rho(U_B) =BU_{B}$. Let $G_{B}$ be the difference Galois group of the system ${\rho(Y)=BY}$ over $K$.
 Our choice of $B$ ensures that $\det(B)=\rho(g)/g$.  Thus, the equation 
$$
\rho(y)=\det(B)y
$$ 
has a solution in $K$. It follows that its difference Galois group $\det (G_B)$ is reduced to $\{1\}$. 
Hence, $G_{B}\subset \mathrm{SL}_{n}(\C)$.  
For $k\geq 1$, we let $G_{B_{[k]}}$ denote the difference 
Galois group of $\rho^{k} (Y)= B_{[k]}Y$ over $K$.  
By Proposition \ref{prop: diffGaloisconnexe}, there exists $r\geq 1$ such that $G_{B_{[r]}}$ is connected. 
Furthermore, $\rho^{r}(g)/g=\det(B_{[r]})$ and the difference Galois group of $\rho^{r} (Y)=B_{[r]}Y$ over $K$  
remains included in $\mathrm{SL}_{n}(\C)$.  Let us notice that 
$$
A_{[r]}=\alpha_{r}B_{[r]}$$ for some $\alpha_{r}\in K^{\times}$.

We claim that $G_{B_{[r]}}$ is irreducible. Let us argue by contradiction, assuming that $G_{B_{[r]}}$ is reducible. 
By Lemma \ref{lem: reducible1}, there exists  $T\in \GL_{n}(K)$ such that 
${\rho^{r}(T)B_{[r]}T^{-1}}$ is a block upper triangular matrix, and we deduce that 
$$
\alpha_{r} \rho^{r}(T)B_{[r]}T^{-1}=\rho^{r}(T)A_{[r]}T^{-1}\, .
$$
By Lemma \ref{lem: reducible1}, we obtain that $G_{r}$, the difference Galois group of the system ${\rho^{r} (Y)=A_{[r]}Y}$ over $K$, 
is conjugated to a group of  block  upper triangular matrices, which implies that  $G_{r}$ is reducible.
By   Remark \ref{rem:groupofiteratesifconnected}, $G_{r}=G$, providing 
a contradiction with the assumption that $G$ is irreducible. This proves that $G_{B_{[r]}}$ is irreducible.  

Finally by Lemma \ref{lem6}, the irreducible connected group $G_{B_{[r]}}$ is primitive. 
Again, since $G_{B_{[r]}}\subset \mathrm{SL}_{n}(\C)$, we infer from \cite[Proposition 2.3]{SU93} that $G_{B_{[r]}}$
 is semisimple. Since $\tcQ_1$ s a pseudo 
$\delta$-field  that contains $\tK$ and all coordinates of $U_B$, we have 
$\tK \langle U_B \rangle \subset \tcQ_1$  (see Section \ref{sec:PPV} for the notation). 
Furthermore,  since $\tcQ_1^{\rho} =\tC$, the pseudo $\delta$-field $\tK \langle U_B \rangle$ is a 
$(\rho^r,\delta)$-Picard-Vessiot extension for 
$
\rho^{r}(Y)=B_{[r]}Y
$  
over $\widetilde{K}$. 

Since all entries of $U_B$ are $\delta$-algebraic over 
 $\widetilde{K}$, we can apply Proposition~\ref{propo:AS} and, arguing as in the first part,  we deduce that 
 $G_{B_{[r]}}=\{I_n\}$. Since $n \geq 2$,  we obtain a contradiction with the fact that $G_{B_{[r]}}$ is irreducible. 
\end{proof}

\subsection{The general case}\label{sec:final}

We are now ready to prove Theorem \ref{thm: main}.  

\begin{proof}[Proof of Theorem \ref{thm: main}]
We argue by induction on $n$. More precisely, our induction assumption reads as follows. 
\begin{itemize}
\item[($\mbox{H}_n$)] For all positive integers $k$ and all $f\in F_{0}$ that is solution to a 
$\rho^{k}$-linear equation of order at most $n$  with coefficients in $K$, we have either  
$f$ is $\tr$-transcendental over $K$ or  $f\in K$. 
\end{itemize}
Proposition \ref{prop: order1} implies that $(\mbox{H}_{1})$ hold true. 
Let $n\geq 2$ and let us assume  
$(\mbox{H}_{n-1})$.  
Let $f\in F_{0}$ be solution to a $\rho^{k}$-linear equation of order $n$ with coefficients in $K$.  
Without any loss of generality, we can assume that $f\neq 0$ and $k=1$. 
Considering the companion matrix associated with this equation,   Corollary~\ref{cor:iterationGaloisgroupconnectedPPVfield} ensures the existence of a positive integer $r$ 
such that the following properties hold.  

\begin{itemize}
\item[(a)] The vector $(f,\dots, \rho^{n-1} (f))^{\top}$ is solution to the system $\rho^{r} (Y)=AY$ for some $A\in \GL_{n}(K)$.  

\item[(b)] There exists a $(\rho^{r},\delta)$-Picard-Vessiot field extension $\tcQ$ for $\rho^r (Y)=AY$ over $\tK$, such that the vector 
$(f,\dots, \rho^{n-1} (f))^{\top}$  is the first column of a fundamental matrix $U\in\GL_{n}(\tcQ)$.
  
\item[(c)]  The difference Galois group $G$ of the system $\rho^{r} (Y)=AY$ over $K$ is connected.
\end{itemize}

If $G$ is irreducible, Proposition \ref{prop: irreducible} shows that $f$ is $\tr$-transcendental. 
Hence $(\mbox{H}_{n})$ holds. 
From now on, we  assume that  $G$ is reducible. 
Furthermore, we assume that $f$ is $\tr$-algebraic over $K$. Thus, it remains to prove that $f\in K$.  
Without loss of generalities, we can assume that $r=1$. 
 
 By Lemma \ref{lem: reducible1}, there exists a gauge transformation ${T=(t_{i,j})\in \GL_{n}(K)}$ such that 
 $$\rho (T)AT^{-1}=\begin{pmatrix}
  A_1& A_{1,2} \\ 
 0 & A_2  
 \end{pmatrix}\, ,$$ 
 where $A_{i}\in \GL_{n_i}(K)$, $n_1+n_2=n$, and $n_2<n$.
 Furthermore, let us assume that $n_1$ is minimal with respect to this property.    

Set  
 \begin{equation}\label{eq: gi}
 (g_{i})_{i\leq n}^{\top}:=T(\rho^{i-1}(f))_{i\leq n}^{\top} \in F_{0}^n\,.
 \end{equation} 
 The vector $\left(g_i\right)_{n_1+1\leq i\leq n}^{\top}\in F_{0}^{n_2}$ is solution to the system 
$\rho(Y)=A_2 Y$. Furthermore, since $f$ is $\tr$-algebraic over $K$, the $g_i$ are also $\tr$-algebraic over $K$. 
By ($\mbox{H}_{n_2}$) and Remark \ref{rem:orderofthesystorderoftheequation}, we obtain that 
\begin{equation}\label{eq: gidansK}
\left(g_i\right)_{n_1+1\leq i\leq n}^{\top}\in K^{n_2}\, .
\end{equation}
Let $G_{1}$ denote the difference 
Galois group of the system 
$
\rho(Y)=A_{1}Y
$
over $K$.  We claim that $G_{1}$ is connected and irreducible. 
Indeed, if $G_{1}$ were reducible, 
then by Lemma \ref{lem: reducible1}, there would exist a gauge transformation changing $A_1$ into a block upper 
 triangular matrix, contradicting the minimality of $n_1$.   Furthermore, the Galois group of $\rho(Y)=A_1Y$ over $K$ is a quotient of the connected group $G$ and thereby a connected group.
 
 The main step of the proof consists in showing the following result. 

\medskip

\noindent\textbf{Claim. } One has $n_1=1$.

\begin{proof}[Proof of the claim]   

By assumption, $f$ is $\tr$-algebraic over $K$. Arguing as in the proof of Proposition \ref{prop: irreducible},  
we get that all coordinates of the vector $(f,\dots, \rho^{n-1} (f))$ are $\delta$-algebraic over $\tK$.  Then, 
\eqref{eq: gi} implies that all the 
$g_{i}$'s are also $\delta$-algebraic over $\widetilde{K}$. 
We have to distinguish two cases.

Let us first assume that all the $g_i$'s belong to $\widetilde{K}$. 
By \eqref{eq: gi} and Lemma~\ref{lem:intersectionF0tK}, we obtain that 
$$
\forall i\in\{1,\ldots, n\}, \;\; g_i\in F_{0}\cap \widetilde{K} =K\,
$$
and hence $(f,\dots,\rho^{n-1}(f))^{\top}=T^{-1}(g_{1},\dots,g_{n})^{\top}\in K^{n}$. 
 Thus the vector $(f,\dots,\rho^{n-1}(f))^{\top}$ is fixed by the difference Galois group $G$.  
 By Lemma \ref{lem: reducible1}, there exists $P\in \mathrm{GL}_{n}(K)$ such that 
$$\rho (P)A P^{-1}=\begin{pmatrix}
  b_{1}& {B}_{1,2} \\ 
 0 & B_2  
 \end{pmatrix}\,,$$ 
for some matrices $B_{1,2}$ and $B_2 $ with coefficients in $K$ and a nonzero $b_{1}\in K$. 
By minimality of $n_{1}$, we obtain that $n_{1}=1$, as wanted.

Now, let us assume that at least one the $g_{i}$'s does not belong to 
$\widetilde{K}$. Note that, by assumption, the $f_i$'s all belong to $\tcQ$ and thus the 
$g_{i}$'s all belong to $\tcQ$ too.  We  let $H$ denote the $(\rho,\delta)$-Galois group of 
$\rho (Y)=AY$ over $\widetilde{K}$. 
By the parametrized Galois correspondence \cite[Theorem 6.20]{HS08}, 
we deduce the existence of some ${\sigma \in H}$, such that 
$$\left( \sigma(g_{i})\right)_{  i\leq n}^{\top}\neq\left(g_{i}\right)_{  i\leq n}^{\top} \, ,
$$ while 
\eqref{eq: gidansK} implies that $\sigma(g_{i})=g_{i}$ for every $i$, $n_1+1\leq i\leq n$. 
Set 
$$
u_1:=\left(g_{i}\right)_{  i\leq n_1}^{\top}, \;\;  u_2:=\left(g_{i}\right)_{  n_1+1\leq i\leq n}^{\top}\mbox{ and } v_1:=\left(\sigma(g_{i})\right)_{  i\leq n_1}^{\top} \,.
$$
Hence $w:=u_1-v_1$ is a nonzero vector. 
Since the coordinates of $u_1$ are $\delta$-algebraic over $\widetilde{K}$ and 
$\sigma$ belongs to the $(\rho,\delta)$-Galois group $H$, 
the coordinates of $v_1$ are also $\delta$-algebraic over $\widetilde{K}$. Hence  
the coordinates of $w$ are  $\delta$-algebraic  over $\widetilde{K}$. 
Furthermore, $u_1$ and $v_1$ are both solution to the system 
$$
\rho(Y)=A_{1}Y+A_{1,2}u_2 \,.
$$
It follows that 
$$
\rho(w) =A_{1}w\,.
$$
Since 
we have 
$$
\widetilde C=\widetilde{K}^\rho\subset \widetilde K(w)^\rho\subset\widetilde{\mathcal{Q}}^\rho =\widetilde C\,,
$$ 
Corollary~\ref{cor:iterationGaloisgroupconnectedPPVfield}
 ensures the existence of a positive integer $s$ and a $(\rho^s, \delta)$-Picard-Vessiot field extension 
$\tcQ_1$ for the system $\rho^s (Y)=(A_{1})_{[s]}Y$ 
over $\widetilde{K}$ such that $w$ is the first column of a fundamental matrix. Furthermore, 
the difference Galois group of   $\rho^s (Y)=(A_{1})_{[s]}Y$ is equal to $G_1$ for the latter is connected. 
The coordinates of $w$ being  $\delta$-algebraic  over $\widetilde{K}$ and $G_{1}$ being 
connected and irreducible, Proposition~\ref{prop: colonne} implies that $n_1=1$. 
\end{proof} 
 
Now, let us prove that $(g_1,\ldots,g_n)\in K^n$. Set 
$$
u_{2}:= \left(g_{i}\right)_{ 2\leq i\leq n}^{\top} \,.
$$
By \eqref{eq: gidansK}, we have $u_{2}\in K^{n-1}$ and thus $A_{1,2}u_{2}\in K^{n-1}$. 
Furthermore, $g_1\in F_{0}$ is solution to  the inhomogeneous order one equation   
$$
\rho (Y)=A_1 Y+A_{1,2}u_{2}\,.
$$
Since $g_1$ is $\tr$-algebraic over $K$, Proposition \ref{prop: order1} implies that $g_1\in K$, and 
hence $(g_1,\ldots,g_n)\in K^n$. 
By \eqref{eq: gi}, we obtain that 
$$
(f,\dots,\rho^{n-1}(f))^{\top}=T^{-1}(g_{1},\dots,g_{n})^{\top}\,.
$$
Since the coefficients of $T$ belong to $K$, it follows that $f\in K$, as wanted.  
\end{proof}

\bibliographystyle{alpha}
\bibliography{biblio_ADH}
\end{document}